\documentclass[11pt]{amsart}
\usepackage{latexsym}

\usepackage{amsmath,amsthm,amsfonts,amscd,eucal,mathabx,amssymb,mathrsfs,pictex}
\usepackage{graphicx}

\usepackage{amssymb}

\usepackage{epstopdf} 

\numberwithin{equation}{section}

\usepackage{amssymb}

\def\ca{{\mathcal A}}
\def\cb{{\mathcal B}}
\def\cc{{\mathcal C}}
\def\cd{{\mathcal D}}

\def\ch{{\mathcal H}}

\def\ck{{\mathcal K}}

\def\cq{{\mathcal Q}}
\def\car{{\mathcal R}}
\def\cs{{\mathcal S}}

\def\ga{{\mathfrak A}}

\def\ba{{\mathbb A}}

\def\bc{{\mathbb C}}

\def\bn{{\mathbb N}}

\def\bq{{\mathbb Q}}
\def\br{{\mathbb R}}
\def\bt{{\mathbb T}}

\def\bz{{\mathbb Z}}

\def\a{\alpha}
\def\b{\beta}
\def\g{\gamma}  \def\G{\Gamma}
\def\d{\delta}  \def\D{\Delta}

\def\eps{\varepsilon}

\def\l{\lambda} 
\def\k{\kappa}
\def\m{\mu}

\def\n{\nu}
\def\r{\rho}
\def\s{\sigma} \def\S{\Sigma}
\def\t{\tau}
\def\f{\varphi}  
\def\th{\theta} 
\def\om{\omega}

\def\id{\hbox{id}}

\newtheorem{thm}{Theorem}[section]
\newtheorem{lem}[thm]{Lemma}

\newtheorem{cor}[thm]{Corollary}
\newtheorem{prop}[thm]{Proposition}

\theoremstyle{definition}

\newtheorem{defin}[thm]{Definition}

\def\!{\mskip-\thinmuskip}

\def\sign{\mathop{\rm sign}}
\def\supp{\mathop{\rm supp}}

\newcommand{\ty}[1]{\mathop{\rm {#1}}}
\def\di{{\rm d}}
\def\id{{\rm id}}

\DeclareMathAlphabet{\mathpzc}{OT1}{pzc}{m}{it}

\begin{document}

\title[noncommutative torus]
{type III representations and modular spectral triples for the noncommutative torus}
\author{Francesco Fidaleo}
\address{Francesco Fidaleo\\
Dipartimento di Matematica \\
Universit\`{a} di Roma Tor Vergata\\
Via della Ricerca Scientifica 1, Roma 00133, Italy} \email{{\tt
fidaleo@mat.uniroma2.it}}
\author{Luca Suriano}
\address{Luca Suriano\\
Dipartimento di Matematica \\
Universit\`{a} di Roma Tor Vergata\\
Via della Ricerca Scientifica 1, Roma 00133, Italy} \email{{\tt
lu.suriano@gmail.com}}
\date{\today}


\begin{abstract}
It is well known that for any irrational rotation number $\a$,  the noncommutative torus $\ba_\a$ must have representations $\pi$ such that the generated von Neumann algebra $\pi(\ba_\a)''$ is of type $\ty{III}$. Therefore, it could be of interest to exhibit and investigate such kind of representations, together with the associated spectral triples whose twist of the Dirac operator and the corresponding derivation arises from the Tomita modular operator. 

In the present paper, we show that this program can be carried out, at least when $\a$ is a Liouville number satisfying a faster approximation property by rationals. In this case,
we exhibit several type $\ty{II_\infty}$ and $\ty{III_\l}$, $\l\in[0,1]$, factor representations and modular spectral triples.

The method developed in the present paper can be generalised  to CCR algebras based on a locally compact abelian group equipped with a symplectic form.
\end{abstract}

\maketitle


\section{Introduction}
\label{sec1}

\noindent

The study of Connes' {\it noncommutative geometry} grew impetuously in the last decades, mostly motivated by possible applications in various fields, among which we first mention Connes' reconstruction theorem in classical geometry, see \cite{C1}. The other potential applications we would like to point out are in physics, in particular to the quantum Hall effect ({\it e.g.} \cite{BS, MP}), and to the attempt to provide a model of equations which unifies the four fundamental interactions known in nature. The reader is referred to \cite{BSvE} and \cite{C0} for the explanation of the possible role of the noncommutative geometry in solving the above mentioned fundamental problems that are still open and not even satisfactorily understood. 

According to Connes' description of the commutative manifolds, the main ingredient in noncommutative geometry is the so called {\it spectral triple}. It is expected that a spectral triple would encode the most important properties of the underlying ``noncommutative manifold'', such as algebraic, geometric, topological, metric, measure-theoretic aspects and many others. For an exhaustive explanation of the fundamental role played by the spectral triples in noncommutative geometry, the reader is referred to the seminal monograph \cite{C}, the expository paper \cite{CPR}, and the literature cited therein. 

In view of other potential applications of noncommutative geometry, very recently it was pursued the program to exhibit twisted spectral triples for which the Dirac operator and the associated derivation are deformed in a suitable way. As explained in the seminal paper \cite{CM}, it would be of interest to implement such a twist by using directly the Tomita modular theory, provided this is nontrivial. The twisted spectral triples considered in the above mentioned paper were called {\it modular}. It is also clear that such modular spectral triples can play a role in constructing new noncommutative geometries in a type $\ty{III}$ setting by emphasising 
the need to take the modular data into account. The modular spectral triples considered in \cite{CM} for the noncommutative 2-torus were constructed by using inner and bounded perturbations of the standard tracial state, therefore providing again type $\ty{II_1}$ von Neumann factors. In fact, the reason for introducing this construction was to obtain 
a curved geometry from the flat one. Concerning this aspect, the reader is referred to \cite{dC, DSW, Ma} for the investigation of the so called Ricci flow and its noncommutative counterpart.

The main goal of the present paper is the preliminary but yet essential step to provide examples of twisted spectral triples, which are
naturally constructed by using the Gel'fand-Naimark-Segal (GNS for short) representation of a $C^*$-algebra $\ga$ relative to a normalised positive linear functional ({\it i.e.} a state) with central support in the bidual, for which the generated von Neumann algebra is a  type $\ty{III}$ factor. For the reader convenience, we recall the basic properties that a modular spectral triple should fulfil.
\begin{itemize}
\item[(i)] The associated twist of the Dirac operator and the relative deformed derivation should come directly from the Tomita  modular operator which, to include type $\ty{III}$ geometries, should be unbounded and not inner ({\it i.e.} it should not come from an inner perturbation of the canonical trace).
\item[(ii)] The Dirac operator suitably deformed with the modular operator as mentioned above, should have compact resolvent at least for the so called "compact geometries".
\item[(iii)] The unbounded deformed derivation associated with such a deformed Dirac operator, acting on the set of bounded operators, should include in its domain sufficiently many elements, {\it e.g.} typically a dense $*$-algebra of the algebra describing the noncommutative manifold under consideration.
\end{itemize}
The construction of modular spectral triples satisfying the requirements listed above will be done in Section \ref {pmstc} for the noncommutative 2-torus.

The previous properties (i)-(iii) are important, since they constitute the minimal requirements for a spectral triple to encode the main metric ({\it i.e.} the Connes distance formula), and measure theoretic properties ({\it i.e.} the noncommutative counterpart of the volume form). We also mention the possible pairing with the K-theory and the equivariant K-theory with the corresponding local index formula.
For such general aspects, the reader is still referred to \cite{CPR, C} and the literature cited therein.

After accepting (i)-(iii) above as the starting point, the natural candidate to construct examples of modular spectral triples will be a state $\om$ on $\ga$ with central support $s(\om)$ in the bidual $\ga^{**}$ of the $C^*$-algebra $\ga$, generating a type $\ty{III}$ representation. Due to the fact that the modular operator will be unbounded and not inner, such a construction appears quite difficult and in fact, to the knowledge of the authors, only few examples of spectral triples satisfying the minimal requirements outlined above are present in literature, even for the noncommutative 2-torus. On the other hand, also type $\ty{III}$ representations for such a pivotal example seem not to be explicitly exhibited in literature. Thus, the other relevant step of the present paper is to partially fill this gap, that is to exhibit examples of such representations for the noncommutative 2-torus.

One of the most studied examples in noncommutative geometry is indeed the noncommutative 2-torus $\ba_\a$ (see {\it e.g.} \cite{B}), since it is a quite simple model, though highly nontrivial. It is related to the discrete Canonical Commutation Relations (CCR for short, see {\it e.g.} \cite{BR}). In fact, it can be considered as a quantum deformation of the classical $2$-torus $\bt^2$, according to the parameter $\a$ corresponding to the rotation of the angle $2\pi\a$ entering in the definition of the symplectic form involved in the construction. When $\a$ is irrational, the $C^*$-algebra $\ba_\a$ is simple and admits a unique, necessarily faithful, trace $\t$. Therefore, the associated GNS representation gives rise to a von Neumann factor $\pi_\t(\ba_\a)''$ which is isomorphic to the Connes hyperfinite type $\ty{II_1}$ factor. It is the noncommutative counterpart of the von Neumann group algebra $L^\infty(\bt^2,m\times m)$ of the $2$-torus, $m$ being the Haar measure, that is the normalised Lebesgue measure on the unit circle. In this case, as $\ba_\a$ is not a type $\ty{I}$ $C^*$-algebra, by the Glimm theorem ({\it cf.} Theorem 1 and Theorem 2 in \cite{G}), it must exhibit type $\ty{III}$ representations.

It is well known that one can associate to the canonical type $\ty{II_1}$ representation, an untwisted spectral triple based on an untwisted Dirac operator chosen in a natural way. We also mention that there are other (untwisted) spectral triples canonically associated to the various complex structures on $\bt^2$, see {\it e.g.} Section 1.4 of \cite{CM}.

Having in mind the various potential applications to type $\ty{III}$ cases, in \cite{CM} it was postulated the existence of twisted spectral triples whose Dirac operator, necessarily twisted, is used to define a twisted commutator. Such twists involve directly the modular data, and the corresponding spectral triples were called "modular". The construction was carried out by deforming the metric ({\it i.e.} the Dirac operator) through an inner and bounded perturbation, directly on the representation associated to the trace. 

However, the most common case for physical applications concerns type $\ty{III_\l}$ von Neumann algebras, $\l\in(0,1]$, describing temperature states at finite inverse temperature $\b\in\br\backslash \{0\}$, apart from the ground state $\b=\infty$ corresponding to the type $\ty{I}$ case and the infinite temperature state $\b=0$ corresponding to the type $\ty{II_1}$ case, see {\it e.g.} \cite{BR}. Recently, in \cite{BF, BF1, F} similar considerations concerning the type of the von Neumann algebras appearing in quantum statistical mechanics have been generalised to disordered systems including the spin glass models.

To conclude with these preliminary observations, we notice that, while the structure of the type $\ty{II_1}$ representation of the noncommutative torus associated to the canonical trace is widely investigated, to the knowledge of the authors the geometries arising from type $\ty{III}$ representations for which the modular data plays a primary role, are not yet satisfactorily studied. Among the works going towards this direction, we would like to mention the recent papers \cite{CPPR, CI, CM, DM, GM, GMR, GMT, Mt}.

As a preliminary step, one of the aim of the present paper is then to exhibit type $\ty{III}$ representations of the noncommutative torus $\ba_\a$, at least in the case when $\a$ is a Liouville number. More precisely, the construction is carried out for any type $\ty{II_\infty}$ and $\ty{III_\l}$, $\l\in[0,1]$ representation. Correspondingly, for such representations we construct the canonically associated spectral triples, suitably deformed by the Tomita modular operator.

Our construction is based upon the explicit knowledge of suitable diffeomorphisms $\mathpzc{f}$ of the circle with rotation number $\r(\mathpzc{f})=\a$ such those considered in \cite{M}. Indeed, by using the method described in \cite{A} to construct non regular representations for the CCR algebras, we shall consider states $\om$ such that the corresponding GNS representations generate von Neumann factors isomorphic to a crossed product, that is
$$
\pi_\om(\ba_\a)''\sim L^\infty(\bt,\di\th/2\pi)\ltimes_{\b}\bz.
$$
For such orientation preserving diffeomorphisms $\mathpzc{f}$ of the unit circle as before, $\b$ is the quasi invariant, free and ergodic action induced on functions $g$ on the unit circle by the corresponding action of $\mathpzc{f}$ on points: $\b(g)=g\circ\mathpzc{f}^{-1}$. For the infinite cases, the type of the factor is then determined by the Krieger-Araki-Woods ratio set relative to the dynamical system $(\bt,\di\th/2\pi,\mathpzc{f})$ generated by all iterates of the diffeomorphism $\mathpzc{f}$ and its inverse $\mathpzc{f}^{-1}$. 

We also note that to each of such diffeomorphisms, there corresponds a probability measure on $\bt$, which is quasi invariant and ergodic under the action of the irrational rotation $2\pi\a$, see below. On one hand, it was proved (see {\it e.g.} \cite{KW, Ke, K4, N, Mo}) the existence of plenty of measures on the unit circle which are quasi invariant and ergodic under the action of the irrational rotations. On the other hand, either the existence of such measures exhibiting any preassigned ratio set is assured by general theorems, or either such measures are explicitly exhibited without any control on the ratio set.
However, examples of such diffeomorphisms are explicitly constructed in \cite{M} for each preassigned kind of ratio set, that is of type $\ty{II_\infty}$ or $\ty{III_\l}$, $\l\in[0,1]$.

Concerning the deformed Dirac operators and the deformed commutators associated to such representations, we prove that such Dirac operators still have compact resolvent, and in addition the deformed commutators still admit a dense $*$-algebra 
in their domain, which is also closed under the entire functional calculus. This is carried out at least in the case when the involved Liouville number $\a$ satisfies a stronger approximation condition by rationals. We note that, in order to provide a real map, our definition of the deformed commutator slightly differs from the analogous one in \cite{CM}.

Our investigation also relies on some results of self containing interest concerning the asymptotic of the growth sequence $\G_n(\mathpzc{f})$ of the circle diffeomorphisms $\mathpzc{f}$ involved in the above construction. We note that, for the diffeomorphisms we shall consider, $\G_n(\mathpzc{f})$ is necessarily unbounded.  It was shown in Theorem 1 of \cite{W} that
$\G_n(\mathpzc{f})=o(n^2)$. On the other hand, in Theorem 2 of the above mentioned paper, for each diverging sequence $a_n=o(n^2)$ there exists a diffeomorphism $\mathpzc{f}$ of the circle without periodic points such that $\G_n(\mathpzc{f})\sim a_n$. Unfortunately, there is no information about the ratio set of the diffeomorphisms constructed in \cite{W}, whereas in \cite{M} no control about the growth sequence is exhibited for examples considered therein. Our result combines both controls on the ratio set and the growth sequence. More precisely, we show that all diffeomorphisms described in \cite{M}, for which we have a precise information about the ratio set, have a growth sequence for which $\G_n(\mathpzc{f})=o(n)$. Furthermore, it is possible to construct diffeomorphims with the same method as in \cite{M} corresponding to a preassigned ratio set, and with growth sequence $\G_n(\mathpzc{f})=o(\ln n)$, provided that the involved Liouville number satisfies the faster approximation condition mentioned above. We also note that our construction of non type $\ty{II_1}$ representations might be carried out also in diophantine case, by relaxing the smoothness properties of the involved diffeomorphisms, see the corresponding discussion in Section \ref{ocfl}.

To end the discussion about the noncommutative geometry relative to the 2-torus, we would like to remark that it admits more classical and well studied noncommutative geometries through the $\th$-deformations studied in \cite{CL}, see {\it e.g.} \cite{Y}. Concerning the general properties of $\th$-deformations, we also mention the relevant reference \cite{BMv}.  

In view of the potential applications of the noncommutative geometry to physics ({\it cf}. \cite{BS, BSvE, C, DM, GMT, Mt, MP}),
we note that the method developed in the present paper can be generalised to arbitrary CCR algebras based on a locally compact abelian group equipped with a symplectic form, see {\it e.g.} \cite{Z}. Concerning a well known particular case, we mention the CCR algebra describing a physical system with $n$ degrees of freedom ({\it i.e.} based on $\br^{2n}\sim\bc^n$ equipped with the canonical symplectic form), for which it is possible to construct type $\ty{III}$ representations in ``apparent contradiction'' with the celebrated Stone-von Neumann uniqueness theorem ({\it cf.} Corollary 5.2.15 in \cite{BR}).

The present paper is organised as follows. The introductory Section 2 contains the preliminary notations and definitions used in the sequel. 

In Section \ref{difa}, we recall the main properties of the diffeomorphisms constructed in \cite{M} for any Liouville number $\a$ such that $\r(\mathpzc{f})=2\a$ (the factor 2 is attached only for the matter of convenience, see below), crucial in our analysis. We also prove two results needed in the sequel, which have a self containing interest. Namely, for the diffeomorphisms 
$\mathpzc{f}$ of the unit circle $\bt$ in Proposition 2.1 of the above mentioned paper, we show that their growth sequence has a not too wild behaviour at infinity: $\G_n(\mathpzc{f})=o(n)$. In addition, if $\a$ satisfies the faster approximation property by rationals {\bf UL} introduced below, one can construct diffeomorphisms as in Proposition 2.1 of  \cite{M} for which $\G_n(\mathpzc{f})=o(\ln n)$. Therefore, for {\bf L} and {\bf UL}-numbers ({\bf L} and {\bf UL} stand for Liouville and Ultra-Liouville, respectively), we have a joint control about the ratio set and the growth sequence of the diffeomorphisms introduced in Proposition \ref{catsum}, the last ones playing a crucial role in constructing non type $\ty{II_1}$ modular spectral triples. 

Section \ref{secstat} is devoted to investigate the properties of a class of states 
$\om_\m$ on the $C^*$-algebra $\ba_{2\a}$ of the noncommutative 2-torus associated to the deformation angle $4\pi\a$. Such states are constructed by considering a probability measure $\m$ on the unit circle $\bt$. In particular, for the rotation $R$ of the angle $2\pi\a$, we show that if $\m\circ R^{n}\sim \m\circ R^{-n}$, $n\in\bz$, then
the support of $\om_\m$ in the bidual $\ba_{2\a}^{**}$ is central ({\it i.e.} the cyclic vector $\xi_\om$ of the GNS representation is also separating for $\pi_{\om_\m}(\ba_{2\a})''$), and therefore $\om_\m$ exhibits a natural modular structure. The GNS representation and the modular data for the states $\om_\m$ as above are then explicitly described in Section \ref{gnsmosc}. 

Section \ref{t3rp} is devoted to determine the type of the von Neumann algebra $\pi_{\om_\m}(\ba_{2\a})''$, which is in fact an hyperfinite factor. In order to construct the state $\om_\m$, for any irrational $\a\in(0,1/2)$ we consider an orientation preserving diffeomorphism $\mathpzc{f}\in C^\infty(\bt)$ with rotation number $\r(\mathpzc{f})=2\a$. Then it is conjugate to the rotation $4\pi\a$ through an essentially unique homeomorphism: $\mathpzc{f}=\mathpzc{h}_\mathpzc{f}\circ\mathpzc{f}\circ\mathpzc{h}_\mathpzc{f}^{-1}$. If $\a$ is diophantine, $\mathpzc{h}_\mathpzc{f}$ is also smooth ({\it cf.} \cite{Yo}), and therefore the unique invariant measure $m\circ\mathpzc{h}_\mathpzc{f}^{-1}$ is equivalent to the Lebesgue measure $m$. Otherwise, if $\a$ is a Liouville number, we can find cases for which 
$m\circ\mathpzc{h}_\mathpzc{f}^{-1}\perp m$ since
$\mathpzc{h}_\mathpzc{f}$ might not be smooth. For all these situations concerning a fixed diffeomorphism $\mathpzc{f}$ as before, the states $\om_\m$ previously described will be associated to the measure $\m:=m\circ\mathpzc{h}_\mathpzc{f}$. We then show that 
$$
\pi_{\om_\m}(\ba_\a)''\sim L^\infty(\bt,\m)\ltimes_{R_{2\a}}\bz\sim
L^\infty(\bt,\di\th/2\pi)\ltimes_{\b}\bz,
$$
which is an hyperfinite von Neumann factor whose type is determined by the ratio set $r([\m],R_{2\a})$, or equivalently by $r([m],\mathpzc{f})$, in the cases when is not of type $\ty{II_1}$. In order to get non type $\ty{II_1}$ representations, it is then enough to look at the diffeomorphisms constructed in \cite{M} for any preassigned ratio set, whose main properties are collected in Section \ref{difa}. 

In Section \ref{dirac}, we study the deformation $D^\s$ of the standard ({\it i.e.} untwisted) Dirac operator $D$ by the Tomita modular operator. We then check whether $D^\s$ has compact resolvent, and show that this happens for all diffeomorphisms $\mathpzc{f}$ of Section \ref{difa} because of the not too wild asymptotic $o(n)$ of their growth sequence $\G_n(\mathpzc{f})$. By following the lines in \cite{CM}, in Section \ref{mstc} we construct the remaining object of the spectral triple, that is the deformed commutator $\cd^\s$ associated to the deformed Dirac operator $D^\s$. Unfortunately, due to an unavoidable obstruction, we see that such a deformed derivation does not admit enough elements providing bounded operators acting on the GNS Hlibert space $\ch_\om$.

The construction of modular spectral triples arising from non type $\ty{II_1}$ representations is carried out in Section \ref{pmstc} by modifying the definition of the underlying Dirac operator. When 
$\a$ is a {\bf UL}-number, by the results in Section \ref{difa}, we can construct diffeomorphisms $\mathpzc{f}$ with a preassigned ratio set, and in addition a better control for their growth sequence ({\it i.e.} $\G_n(\mathpzc{f})=o(\ln n)$). This allows us to construct representations of preassigned type, and in addition Dirac operators with compact resolvent, and deformed commutators such that there are sufficiently many elements in their domains.

The paper is complemented by two further sections. For the convenience of the reader, Section \ref{ocfl} collects some open interesting problems, not only directly connected to noncommutative geometry. Section \ref{ppa} contains the explicit computation of the (perhaps expected) fact that the dense $*$-algebra for which the deformed commutator provides bounded operators, can be chosen to be stable under the entire functional calculus.

\section{Preliminaries}
\label{sec2}

\noindent

In the present section, we collect all notations and definitions which are useful in the paper.

\subsection{Preliminary notations}

Let $E$ be a normed space. We simply denote by $\|\,{\bf\cdot}\,\|$ its norm whenever no confusion arises. In particular, $\|x\|\equiv\|x\|_\ch$ will be the Hilbertian norm of the element $x\in\ch$ in the Hilbert space $\ch$. For a continuous or measurable function $f$ defined on the locally compact space $X$ equipped with the Radon measure $\m$, $\|f\|\equiv\|f\|_\infty$ will denote the ``sup'' or ``esssup'' norm of the function $f$, respectively. In addition, for an arbitrary set $X$ ({\it i.e.} equipped with the discrete topology), $\cb(X)$ will denote the $C^*$-algebra consisting of all bounded functions defined on it, equipped with the natural algebraic operations and norm $\|f\|\equiv\|f\|_\infty:=\sup_{x\in X}|f(x)|$.

Let $X$ be a locally compact abelian group together with its dual group $\widehat{X}$ ({\it i.e.} the group of continuous characters equipped with pointwise multiplication, reciprocal and the topology of the uniform convergence on the compact subsets of $X$), with the Haar measures $\di x, \di\chi$, respectively. The Fourier transform and anti-transform of a function $f\in L^1(X, \di x)$ are defined as
$$
\hat f(\chi):=\int_X f(x)\overline{\chi(x)}\di x,\quad \check f(\chi):=\int_X f(x)\chi(x)\di x.
$$
By the Riemann-Lebesgue lemma, if $a,b\in C^*(X)$, the enveloping $C^*$-algebra of the $L^1$-convolution algebra of $X$, their convolution is a well defined element $a*b\in C^*(X)$, whose corresponding Fourier transform and anti-transform are functions in $C_o(\widehat{X})$, the space of all continuous functions on $\widehat{X}$ vanishing at infinity,
satisfying
$$
\widehat{a*b}(\chi)=\hat a(\chi)\hat b(\chi),\quad \widecheck{a*b}(\chi)=\check a(\chi)\check b(\chi).
$$
Let $\bt:=\{z\in\bc\mid |z|=1\}$ be the abelian group consisting of the unit circle. It is well known that the dual topological group $\widehat{\bt}$ is isomorphic to the discrete group $\bz$. The corresponding Haar measures are $\di m=\frac{\di z}{2\pi\imath z}=\frac{\di \th}{2\pi}$ for $z=e^{\imath\th}$ on the circle, and the counting measure on $\bz$, respectively.

Let $(X,\cb)$ be a measurable space, together with two $\s$-additive measures $\m$ and $\n$ on the $\s$-algebra $\cb$, which are supposed to be always $\s$-finite and positive. If 
$\m$ dominates $\n$ in the sense of measures ({\it e.g.} Section I.4 of \cite{RS}), we write $\n\preceq\m$. If $\n\preceq\m$ and $\m\preceq\n$, then $\m$ and $\n$ are equivalent as measures, and we write $\m\sim\n$.

For an essentially bounded measurable function $f$ on the measure space $(X,\cb,\n)$, the multiplication operator $M_f$ is the closed operator acting on $L^2(X,\cb,\n)$ with domain
$$
\cd_{M_f}:=\bigg\{g\in L^2(X,\cb,\n)\mid \int_X\big|f(x)g(x)\big|^2\di\n(x)<+\infty\bigg\},
$$
defined for $g\in\cd_{M_f}$ as
$$
(M_fg)(x):=f(x)g(x),\quad x\in X.
$$

\subsection{On measurable automorphisms}
\label{5dfm}

To simplify notation, we will denote by $R$ the rotation by the angle $2\pi\a$ ({\it i.e.} $R\equiv R_{\a}$) whenever the parameter $\a$ is fixed. Set $\m_n:=\m\circ R^{-n}$ for each $n\in\bz$. 

More generally, if $\mathpzc{f}:Y\to Y$ is an invertible map on the set $Y$:
\begin{itemize} 
\item $\mathpzc{f}^0:=\id_Y$, and its inverse is denoted by $\mathpzc{f}^{-1}$;
\item for the $n$-times composition, $\mathpzc{f}^n:=\underbrace{\mathpzc{f}\circ\cdots\circ \mathpzc{f}}_{n-\text{times}}$;
\item for the $n$-times composition of the inverse, $\mathpzc{f}^{-n}:=\underbrace{\mathpzc{f}^{-1}\circ\cdots\circ \mathpzc{f}^{-1}}_{n-\text{times}}$.
\end{itemize}
Thus, $\mathpzc{f}^{n}$ is meaningful for any $n\in\bz$ with the above convention. 

Recall that an automorphism $\mathpzc{f}$ of a standard measure space $(X,\cb,\n)$ is a bi-measurable bijection between two conull subsets of $X$ such that the pull-back measure satisfies $\n\circ\mathpzc{f}\sim\n$. The automorphism $\mathpzc{f}$ is called {\it ergodic} if any $\mathpzc{f}$-invariant measurable subset of X is either null or conull. When the measure space $(X,\cb,\n)$ is fixed, we will just speak about an ergodic automorphism $\mathpzc{f}$. If it is not of type $\ty{I}$ ({\it i.e.} $\n$ is not purely atomic), any such an automorphism is further classified by the Krieger-Araki-Woods {\it ratio set} ({\it cf.} \cite{K2}), denoted by $r(\mathpzc{f})\subset[0,+\infty)$ when the measure class $[\n]$ determined by $\n$ is fixed, or more precisely by $r([\n],\mathpzc{f})$. 

Consider a non type $\ty{I}$ automorphism $\mathpzc{f}$ as above. It is said to be of 
\begin{itemize}
\item[(ii)] type $\ty{II}$, if $r(\mathpzc{f})=\{1\}$.
\end{itemize}
Notice that this is the case if and only if $\n\sim\n_0$ for a (essentially unique) $\s$-finite invariant measure $\n_0$. It is said to be of
\begin{itemize}
\item[(ii${}_1$)] type $\ty{II_1}$, if $\n_0$ is finite,
\item[(ii${}_\infty$)] type $\ty{II_\infty}$, if $\n_0$ is infinite.
\end{itemize}
In the remaining cases, it is said to be of 
\begin{itemize}
\item[(iii)] type $\ty{III}$, if it is not of type $\ty{II}$.
\end{itemize}
Furthermore, it is said to be of
\begin{itemize}
\item[(iii${}_0$)] type $\ty{III_0}$, if $r(\mathpzc{f})=\{0,1\}$,
\item[(iii${}_\l$)] type $\ty{III_\l}$ for some $\l\in(0,1)$, if $r(\mathpzc{f})=\l^\bz\cup\{0\}$,
\item[(iii${}_1$)] type $\ty{III_1}$, if $r(\mathpzc{f})=[0,+\infty)$.
\end{itemize}
Consider the natural action of $\bz$ on $L^\infty(X,\n)$ generated by the automorphism $\mathpzc{f}$: $\b(g):=g\circ\mathpzc{f}^{-1}$ for $g\in L^\infty(X,\n)$.
For the crossed product $L^\infty(X,\n)\ltimes_{\b}\bz$, it is well known that ${\rm S}\left(L^\infty(X,\n)\ltimes_{\b}\bz\right)=r(\mathpzc{f})$, where ${\rm S}(M)$ is the Connes spectral invariant of the factor $M$, see {\it e.g.} \cite{RMPR, K1, K2a, K3}. 

We now specialise the matter to the unit circle. In this case, suppose that for some real number $\a\in(0,1)$, $\mathpzc{f}=\mathpzc{h}\circ R_\a\circ \mathpzc{h}^{-1}$ for some homeomorphism $\mathpzc{h}$ of the unit circle. Then $\r(\mathpzc{f})=\a$, where $\r(\mathpzc{f})$ is the {\it rotation number} ({\it e.g.} Section 11 of \cite{KH}) of the homeomorphism $\mathpzc{f}$. Conversely, for $\a$ irrational, suppose that $\mathpzc{f}$ is an orientation preserving $C^1$-diffeomorphism with bounded variation derivative, such that $\r(\mathpzc{f})=\a$. Then 
\begin{itemize}
\item[{\bf (A)}] $\mathpzc{f}$ is topologically conjugate to the rotation $R_\a$, that is
$$
\mathpzc{f}=\mathpzc{h}_\mathpzc{f}\circ R_\a\circ \mathpzc{h}_\mathpzc{f}^{-1},
$$
for a unique homeomorphism $\mathpzc{h}_\mathpzc{f}$ of $\bt$ with $\mathpzc{h}_\mathpzc{f}(1)=1$.
\end{itemize}
If $\a$ is diophantine and $\mathpzc{f}$ is a sufficiently smooth diffeomorphism, $\mathpzc{h}_\mathpzc{f}$ is indeed smooth, see {\it e.g.} \cite{Yo}.

The following fact is immediately true.
\begin{lem}
\label{equ}
Suppose that a homeomorphism $\mathpzc{f}$ of the circle $\bt$ is conjugate to a rotation $R_\a$ according to {\rm{({\bf A})}}.  
Then the homeomorphism of the circle defined as
$$
\mathpzc{g}:=\mathpzc{h}_\mathpzc{f}\circ R_{\a/2}\circ\mathpzc{h}_\mathpzc{f}^{-1}
$$
is a ``square root'' of $\mathpzc{f}$: $\mathpzc{g}\circ \mathpzc{g}=\mathpzc{f}$.
\end{lem}
Given an irrational number which is not diophantine, it is possible to prove that it is a Liouville number, that is it satisfies the condition {\bf (L)} below.
\begin{itemize}
\item[{\bf (L)}] A {\it Liouville number} $\a\in(0,1)$ is a real number such that for each $N\in\bn$ the inequality 
\begin{equation*}
\left|\a - \frac{p}{q}\right|< \frac{1}{q^N}
\end{equation*}
has an infinite number of solutions for $p,q\in\bn$ with $(p,q)=1$.
\item[{\bf (UL)}] Among Liouville numbers, we also consider those such that, for each $\l>1$ and $N\in\bn$, the inequality
\begin{equation*}
\left|\a - \frac{p}{q}\right|<\frac{1}{\l^{q^N}}
\end{equation*}
again admits infinite number of solutions for $p,q\in\bn$ with $(p,q)=1$.
\end{itemize}
For details and proofs, we refer the reader to \cite{KH, Ki} and the reference cited therein. 

Let $\frac{\di\,}{\di\th}=\imath z\frac{\di\,}{\di z}$ be the derivative w.r.t. the angle $\th$. If the ``natural'' variable of $\mathpzc{f}$ is $z$, then (with a little abuse of notation) we simply denote by $D$ the derivative w.r.t. its argument ({\it i.e.} $D\equiv\frac{\di\,}{\di z}$). We refer to a $C^\infty$-diffeomorphism $\mathpzc{f}:\bt\rightarrow\bt$ simply as a ``diffeomorphism'' if is not otherwise specified.

Given a $C^1$-diffeomorphism $\mathpzc{f}$ of the unit circle $\bt$, the {\it growth sequence} $\{\G_n(\mathpzc{f})\mid n\in\bn\}$ is defined as
$$
\G_n(\mathpzc{f}):=\|D\mathpzc{f}^n\|_\infty\vee\|D\mathpzc{f}^{-n}\|_\infty,\quad n\in\bn.
$$
We set
$$
\g(\mathpzc{f}):=\sup_{n\in\bn}\G_n(\mathpzc{f})\in\br_+\cup\{+\infty\}.
$$
If $\mathpzc{f}$ is $C^1$-conjugate to the rotation by some angle $2\pi\a$, then $\g(\mathpzc{f})<+\infty$. As pointed out before, this happens for $C^\infty$-diffeomorphisms $\mathpzc{f}$ with 
$\r(\mathpzc{f})=\a$ whenever $\a$ is diophantine.

Suppose that $\b\in(0,1)\backslash\bq$ and $\mathpzc{f}$ is an orientation preserving diffeomorphism with $\r(\mathpzc{f})=\b$. In this case, the Denjoy Theorem asserts that there exists a unique homeomorphism $\mathpzc{h}_\mathpzc{f}$ of the unit circle such that $\mathpzc{h}_\mathpzc{f}(1)=1$ satisfying {\bf (A)} for the rotation $R_\b$. Then
\begin{equation}
\label{uinma}
\m_\mathpzc{f}:=(\mathpzc{h}_\mathpzc{f})^*m=m\circ \mathpzc{h}_\mathpzc{f}^{-1},
\end{equation}
is the unique invariant measure, which is ergodic for the natural action of $\mathpzc{f}$ on $\bt$. For a diophantine number $\b$, $\mathpzc{h}_\mathpzc{f}$ is indeed smooth and thus $\m_\mathpzc{f}\sim m$. For a Liouville number $\b$, things are quite different. There are diffeomorphisms as above for which the unique invariant measure $(\mathpzc{h}_\mathpzc{f})^*m$ is singular w.r.t. the Haar measure $m$: $(\mathpzc{h}_\mathpzc{f})^*m\perp m$. In order to exhibit type $\ty{III}$ representations of the noncommutative torus $\ba_{2\a}$,
we will actually look at such diffeomorphisms,  where $\b=2\a$. The use of the factor 2 is pure matter of convenience, and will be clarified in the sequel. 

To deal with Dirac operators on $\ba_{2\a}$, we must also look at the Haar measure $m$ on $\bt$. In fact, we begin by noticing that the sequence of the measures $\{m\circ\mathpzc{f}^n\mid n\in\bz\}$ are all quasi equivalent, as $\mathpzc{f}$ is a diffeomorphism. We then compute the corresponding Radon-Nikodym derivatives in the following
\begin{lem}
For any $C^1$-diffeomorphism of the circle $\mathpzc{f}$, we have
\begin{equation}
\label{rcanz}
\frac{\di m\circ\mathpzc{f}^{-n}}{\di m}(z)=\frac{z(D\mathpzc{f}^{-n})(z)}{\mathpzc{f}^{-n}(z)},\quad n\in\bz.
\end{equation}
\end{lem}
\begin{proof}
The proof relies on an elementary change of variable. In fact, fix any function $h\in C(\bt)$, and set $\mathpzc{g}=\mathpzc{f}^n$. We then have
\begin{align*}
&\int_{\bt}\di m\circ\mathpzc{g}^{-1}(z)h(z)=\int_{\bt}\di m(z)h(\mathpzc{g}(z))=\int_{\bt}\frac{\di z}{2 \pi\imath z}h(\mathpzc{g}(z))\\
=&\int_{\bt}\frac{\di w}{2 \pi\imath w}\frac{wh(w)}{\mathpzc{g}^{-1}(w)(D\mathpzc{g})(\mathpzc{g}^{-1}(w))}
=\int_{\bt}\frac{\di z}{2 \pi\imath z}\frac{z(D\mathpzc{g}^{-1})(z)}{\mathpzc{g}^{-1}(z)}h(z)\\
=&\int_{\bt}\frac{\di z}{2 \pi\imath z}\frac{z(D\mathpzc{f}^{-n})(z)}{\mathpzc{f}^{-n}(z)}h(z)\,.
\end{align*}
\end{proof}
We also notice the crucial fact that the Lebesgue measure $m$ is ergodic under the action of the $C^\infty$-diffeomorphisms $\mathpzc{f}$ considered in the present paper, see {\it e.g.} Theorem 12.7.2 in \cite{KH}.

\subsection{The noncommutative 2-torus}

For a fixed $\a\in\br$, the {\it noncommutative torus} $\ba_{2\a}$ associated to the rotation by the angle $4\pi\a$, is the universal unital $C^*$-algebra with identity $I$ generated by the commutation relations involving two noncommutative unitary indeterminates $U,V$:
\begin{equation}
\label{ccrba}
\begin{split}
&UU^*=U^*U=I=VV^*=V^*V,\\
&UV=e^{4\pi\imath\a}VU.
\end{split}
\end{equation}

We express $\ba_{2\a}$ in the so called {\it Weyl form}. Let ${\bf a}:=(m,n) \in \bz^2$ be a double sequence of integers, and define
$$
W({\bf a}):=e^{-2\pi\imath\a mn}U^mV^n,\quad {\bf a}\in\bz^2.
$$
Obviously, $W({\bf 0})=I$, and the commutation relations \eqref{ccrba} become
\begin{equation}
\label{ccrba1}
\begin{split}
&W({\bf a})W({\bf A})=W({\bf a}+{\bf A})e^{2\pi\imath\a\s({\bf a},{\bf A})},\\
&W({\bf a})^*=W(-{\bf a}),\quad {\bf a}, {\bf A}\in\bz^2,
\end{split}
\end{equation}
where the symplectic form $\s$ is defined by
$$
\s({\bf a},{\bf A}):=(mN-Mn),\quad {\bf a}=(m,n), \,{\bf A}=(M,N) \in\bz^2.
$$
We now fix a function $f\in\cb(\bz^2)$, which we may assume to have finite support. The element $W(f)\in\ba_{2\a}$ is then defined as
$$
W(f):=\sum_{{\bf a}\in\bz^2}f({\bf a})W({\bf a}).
$$
The set $\{W(f)\in\ba_{2\a}\mid f\in\cb(\bz^2)\,\text{with finite support}\}$ provides a dense $*$-algebra of $\ba_{2\a}$. 

For $\a$ irrational, we recall that $\ba_{2\a}$ is simple and has a necessarily unique faithful trace $\t$ given by
$$
\t(W(f)):=f({\bf 0}),\quad W(f)\in\ba_{2\a}.
$$
Conversely, by Remark 1.7 of \cite{B}, any element $A\in\ba_{2\a}$ is uniquely determined by the corresponding Fourier coefficients
\begin{equation}
\label{zfua}
f({\bf a}):=\t(W(-{\bf a})A),\quad {\bf a}\in\bz^2.
\end{equation}
The relations \eqref{ccrba1} transfer to the generators $W(f)$ as follows:
$$
W(f)^*=W(f^\star),\quad W(f)W(g)=W(f*_{2\a} g),
$$
where
\begin{align*}
&f^\star({\bf a}):=\overline{f(-{\bf a})},\\
&(f*_{2\a} g)({\bf a})=\sum_{{\bf A}\in\bz^2}f({\bf A})g({\bf a}-{\bf A})e^{-2\pi\imath\a\s({\bf a},{\bf A})}.
\end{align*}
It is then easily seen that, for each fixed $n\in\bz$, $f^{(n)}(m):=f(m,n)$ defines a sequence whose Fourier anti-transform 
$$
\widecheck{f^{(n)}}(z):=\sum_{m\in\bz}f(m,n) z^m
$$
provides a continuous function $\widecheck{f^{(n)}}\in C(\bt)$.

From now on, we tacitly assume that $\a\in(0,1/2)$ is always irrational. Moreover, in order to get type $\ty{III}$ representations and the corresponding modular spectral triples, we shall restrict the matter to Liouville numbers {\bf (L)}, possibly satisfying the stronger approximation condition {\bf (UL)} when necessary, even if the general part of the present analysis works for general irrational numbers 
$\a\in\br$.

\subsection{Modular spectral triples}

Concerning the usual terminology, the main concepts and results in operator algebras theory such as Tomita modular theory and so forth, the reader is referred to \cite{BR, SZ, T1} and the reference cited therein. 

For the convenience of the reader, we report the following well known fact whose proof can be found in pag. 15 of \cite{NSZ}.
\begin{prop}
\label{ccbl}
Let $\f\in\ga^*_+$ be a positive linear functional on the $C^*$-algebra $\ga$, together with its GNS representation $(\ch_\f,\pi_\f,\xi_\f)$. 

Then $\xi_\f$ is cyclic for the commutant $\pi_\f(\ga)'$ if and only if the support $s(\f)$ in $\ga^{**}$ is central:  $s(\f)\in Z(\ga^{**})$.
\end{prop}
For the $C^*$-algebra $\ga$, we denote with $\cs(\ga)\subset\ga^*_+$ the set of the states, that is all positive normalised functionals on $\ga$.

Let $\om\in\cs(\ga)$ with $s(\om)\in Z(\ga^{**})$. Then
$$
\pi_\om(a)\xi_\om\in\ch_\om\mapsto\pi_\om(a^*)\xi_\om\in\ch_\om
$$ 
is well defined on the dense subset $\big\{\pi_\om(a)\xi_\om\mid a\in\ga\big\}\subset\ch_\om$, and closable. The polar decomposition of its closure $S_\om$, named the Tomita involution, is usually written as $S_\om=J_\om\D_\om^{1/2}$, where $J_\om$ and $\D_\om$ are the Tomita conjugation and modular operator, respectively. We shall drop the subscripts whenever the state $\om$ is fixed once for all.

We introduce the definition of (even) modular spectral triple that we will use in the sequel, which is slightly different from the analogous one in \cite{CM}. For general aspects and some natural applications of spectral triples, we refer the reader to \cite{CPR, C, V} and the references cited therein.
\begin{defin}
\label{fmst}
A {\it modular spectral triple} associated to a unital $C^*$-algebra $\ga$ is a triplet $(\om,\ca,L)$, where $\om\in\cs(\ga)$, $\ca\subset\ga$ is a dense $*$-algebra and $L$ is a densely defined closed operator acting on $\ch_\om$, satisfying the following conditions:
\begin{itemize}
\item[(i)] $s(\om)\in Z(\ga^{**})$;
\item[(ii)] the deformed Dirac operator 
$$
D^\s:=\begin{pmatrix} 
	 0 &\D_\om^{-1}L\\
	L^*\D_\om^{-1}& 0\\
     \end{pmatrix} 
$$
acting on $\ch_\om\oplus\ch_\om$ uniquely defines a selfadjoint operator with compact resolvent: $D^\s$ is densely defined essentially selfadjoint 
with $\big(1+(\overline{D^\s})^2\big)^{-1/2}$ compact;
\item[(iii)] for each $a\in\ca$, the deformed commutator
$$
\cd_L^\s\big(\pi_\om(a)\big):=\imath\begin{pmatrix} 
	0&\D_\om^{-1}[L,\pi_\om(a)]\\
	[L^*,\pi_\om(a)]\D_\om^{-1}& 0\\
     \end{pmatrix}
$$
uniquely defines a bounded operator: $\overline{\cd_{\cd_L^\s(\pi_\om(a))}}=\ch_\om\oplus\ch_\om$ and 
$$
\sup\big\{\big\|\cd_L^\s\big(\pi_\om(a)\big)\xi\big\|\mid \xi\in\cd_{\cd_L^\s(\pi_\om(a))},\,\,\|\xi\|\leq1\big\}<+\infty;
$$
\item[(iv)] $\pi_\om(\ca)\cd_L\subset \cd_L$, $\pi_\om(\ca)\cd_{L^*}\subset \cd_{L^*}$.
\end{itemize}
\end{defin}
\vskip.5cm
The closure of the Dirac operator in (ii) will be also denoted as $D^\s$ with an abuse of notation.

Concerning the question of whether a spectral triple determines a {\it Fredholm module} ({\it e.g.} \cite{C}), it is customary to add the additional condition (iv) in the definition of spectral triple. We would like to point out that, although the situation is well clarified in the untwisted case in \cite{FMR}, in the situation treated in the present paper it is still unclear what should be the conditions to obtain a Fredholm module from a modular spectral triple fulfilling the properties in Definition \ref{fmst}. Indeed, it might be necessary to use a definition that is based on some suitable twisting. We do not want to pursue to much these relevant aspects, and postpone a more detailed analysis somewhere else.

We end by pointing out that, even if the terminology ``modular spectral triple" is used in slightly different contexts ({\it e.g.} \cite{CPPR}), we still call ``modular" a spectral triple fulfilling the properties in Definition \ref{fmst} in accordance to the analogous one in Section 1 of \cite{CM}, to which our definition is widely close. The reader is referred to Section \ref{mstc} for the explanation of the appearance of condition (iii) above, which is slightly different from the analogous one in the above mentioned paper \cite{CM}.

\section{On the diffeomorphisms on the circle}
\label{difa}

\noindent
In the present section, we report some useful results listed in Proposition 2.1 of \cite{M} concerning diffeomorphisms of the unit circle topologically conjugate to a rotation by $2\pi$-times a Liouville number $\a$, that is a number satisfying condition {\bf(L)}. In the previously mentioned paper, the diffeomorphisms $\mathpzc{f}$ are constructed in order to have a control on the ratio set $r([m],\mathpzc{f})$, but without any control on the growth sequence, the latter being crucial in exhibiting modular spectral triples associated with non type $\ty{II_1}$ representations. On the other hand, the asymptotic of the growth sequence is studied in several papers, and in particular in the case of sufficiently smooth diffeomorphisms without periodic points, in \cite{W}. Contrary to \cite{M}, no control about the ratio set is investigated therein. For such diffeomorphisms, we provide also a control on the growth sequence. We then extend the topic to diffeomorphisms 
$\mathpzc{f}$ with $\r(\mathpzc{f})=\a$, such that $\a$ satisfies the condition {\bf(UL)}, by exhibiting examples with a better behaviour of their growth sequence.

By using the terminology of \cite{M}, to which we refer the reader for further details, we also collect some of the points listed in Proposition 2.1 in the above paper (see also Lemma 4.1 in \cite{FS}), which are crucial for the purpose of our analysis, in the following
\begin{prop}
\label{catsum}
Let $\a\in(0,1)$ satisfy {\rm{\bf(L)}}, and consider any type  $\ty{II_\infty}$, or $\ty{III_\l}$, $\l\in[0,1]$ diffeomorphism of $\bt$, together with the corresponding ratio set ${\rm F}\subset[0,+\infty)$. 

Then there exist sequences of rational numbers $\a_n=\frac{p_n}{q_n}\in(0,1)$ with $(p_n,q_n)=1$ for each $n\in\bn$, and diffeomorphisms of the circle $h_n$ such that, by setting $H_0:=\id_\bt$, $\mathpzc{f}_0:=R_{\a_1}$, and defining recursively 
$$
H_n:=h_1\cdots h_n,\quad \mathpzc{f}_n:=H_n\circ R_{\a_{n+1}}\circ H_n^{-1},\,\,n\in\bn,
$$
\begin{itemize}
\item[(i)] $\a_n\to\a$.
\item[(ii)] $R_{\a_{n}}\circ h_n= h_n\circ R_{\a_{n}}$, $n\in\bn$.
\item[(iii)] $H_n$ converges uniformly to a homemorphism $H$.
\item[(iv)] $\mathpzc{f}_n$ converges in the $C^{\infty}$-topology to a diffeomorphism $\mathpzc{f}$ satisfying 
$$
\mathpzc{f}=H\circ R_{\a}\circ H^{-1},\,\, (\text{hence}\,\, \r(\mathpzc{f})=\a).
$$
\item[(v)] In the type $\ty{II_\infty}$ case, the Lebesgue measure $m$ on $\bt$ is equivalent to an infinite measure which is invariant under the action of $\mathpzc{f}$, and thus, consequently, $r(\mathpzc{f})=\{1\}\equiv{\rm F}$. 
In the type $\ty{III}$ case, $r(\mathpzc{f})={\rm F}$.
\item[(vi)] $\G_n(\mathpzc{f})=o(n)$.
\item[(vii)] If in addition $\a$ satisfies {\rm{\bf(UL)}}, then the sequences $\{\a_n\}_{n\in\bn}$ and $\{h_n\}_{n\in\bn}$ can be chosen such that the diffeomorphism in {\rm(iv)} satisfies $\G_n(\mathpzc{f})=o(\ln n)$.
\end{itemize}
\end{prop}
\begin{proof}
(i)-(v) are proven in \cite{M}. 

(vi) Let $\{\mathpzc{g}_{l}\mid l\in\bn\}$ be any sequence of $C^1$-diffeomorphisms of the circle, and $m>l$. We get
\begin{align*}
&\|D\mathpzc{g}_m-D\mathpzc{g}_l\|\leq\sum_{k=l+1}^m\|D\mathpzc{g}_k-D\mathpzc{g}_{k-1}\|\\
\leq&\sum_{k=l+1}^m\big(\|D\mathpzc{g}_k-D\mathpzc{g}_{k-1}\|\vee\|D\mathpzc{g}^{-1}_k-D\mathpzc{}^{-1}_{k-1}\|\big)\\
=&\sum_{k=l+1}^m d_1(\mathpzc{g}_k,\mathpzc{g}_{k-1})\leq\sum_{k=l+1}^{+\infty} d_1(\mathpzc{g}_k,\mathpzc{g}_{k-1}).
\end{align*}
If $\mathpzc{g}_m$ converges together with its $1^{{\rm st}}$-derivative to the diffeomorphism $\mathpzc{g}$, by taking the limit on the l.h.s., we get
$$
\|D\mathpzc{g}-D\mathpzc{g}_l\|\leq \sum_{k=l+1}^{+\infty} d_1(\mathpzc{g}_k,\mathpzc{g}_{k-1}).
$$
Let $\mathpzc{g}=\mathpzc{f}^n$, $n\in\bz$, with $\mathpzc{f}$ any diffeomorphism as in (iv), constructed
in Proposition 2.1 in \cite{M}. With $C(1)$ the constant appearing in Lemma 2.3 of \cite{M} when $r=1$, by Lemma 2.3 in that paper, we obtain
\begin{equation}
\label{2cz1a}
\begin{split}
&d_1(\mathpzc{f}^n_k,\mathpzc{f}^n_{k-1})=d_1\big(H_kR^n_{\a_k}H^{-1}_k,H_kR^n_{\a_{k+1}}H^{-1}\big)\\
=&d_1\big(H_kR_{n\a_k}H^{-1}_k,H_kR_{n\a_{k+1}}H^{-1}\big)\\
\leq&C(1)|\!|\!| H_k|\!|\!|^2_{2}|n||\a_k-\a_{k+1}|.
\end{split}
\end{equation}
By the definition of Liouville number, we can choose (see pag. 1405 of \cite{M}) an approximation $p_k/q_k\to\a$ such that
\begin{equation}
\label{1cz1a}
C(1)|\!|\!| H_k|\!|\!|^2_{2}|\a_k-\a_{k+1}|\leq 2C(1)|\!|\!| H_k|\!|\!|^2_{2}|\a-\a_{k}| \leq \frac{1}{2^{k}}.
\end{equation}
By inserting \eqref{1cz1a} in \eqref{2cz1a}, we get
$$
d_1(\mathpzc{f}^n_k,\mathpzc{f}^n_{k-1})\leq\frac{|n|}{2^{k}},\quad n\in\bz,\,k=1,2.\dots\,\,.
$$
Moreover, 
$$
\mathpzc{f}^n_k\to \mathpzc{f}^n,\,\, D\mathpzc{f}^n_k\to D\mathpzc{f}^n,\quad n\in\bz,
$$
and thus
$$
\|D\mathpzc{f}^n-D\mathpzc{f}^n_k\|\leq \sum_{l=k+1}^{+\infty} d_1(\mathpzc{f}^n_l,\mathpzc{f}^n_{l-1})=\frac{|n|}{2^{k}},\quad n\in\bz,\,\,k\in\bn.
$$
By using again the triangle inequality, we argue
\begin{equation*}
\|D\mathpzc{f}^n\|\leq\frac{|n|}{2^{k}}+\|D\mathpzc{f}^n_k\|\leq\frac{|n|}{2^{k}}+\g(\mathpzc{f}_k)<+\infty,\quad n\in\bz,\,\,k\in\bn,
\end{equation*}
since $\mathpzc{f}_k$ is conjugate to a rational rotation through a diffeomorphism. 

Choose now any strictly increasing sequence of integers $\{k_m\mid m\in\bn\}$ such that $\g(\mathpzc{f}_{k_m})\leq m^{1/2}$. Then
$$
|n|^{-1}\|D\mathpzc{f}^n\|\leq2^{-k_{|n|}-2}+|n|^{-1/2}\to0.
$$

(vii) First we note that, if {\bf(UL)} is satisfied, then $\a$ is {\it a fortiori} a Liouville number, hence conditions (3.4) and (3.5) in \cite{M} (needed  to construct beforehand a sequence of diffeomorphisms $\hat h_n$ such that the diffeomorphism $\mathpzc{f}$ has the desired type, {\it i.e.} $\ty{II_\infty}$ or $\ty{III_\l}, \l\in[0,1]$), are automatically fulfilled.

Assume {\bf(UL)} with $2N$ instead of $N$. For each $N\in\bn$, $\eps>0$ and $C>1$, a simple computation then yields that, if 
$$
\l>e^{C^2}\left(\frac{q^{2N}}{\eps^2}\right)^{\frac1{2q^{2N}}},
$$
then the inequality
\begin{equation}
\label{uLn}
 \left|\a - \frac{p}{q}\right| <\frac{\eps}{q^N e^{(C q^N)^2}}
\end{equation}
has an infinite number of solutions with $(p,q)=1$. Hence, for such an $\a$ we can immediately assume the latter condition and perform the same reasonings as in \cite{M}.

By (2.4) in \cite{M}, we have
$$
\|D\mathpzc{f}_k^n-D\id_\bt\| \leq |\!|\!|\mathpzc{f}_k^n|\!|\!|_1 \leq |\!|\!|H_k|\!|\!|_2^2 \leq C(k,2)^2 q_k^{2 N(k,2)}.
$$
Therefore, setting
$$
\eps=\frac{1}{2^{k+1} C(1)C(k,2)^2}, \quad N=2 N(k,2), \quad C=C(k,2)^2,
$$
and
\begin{equation}
\label{seq}
m_{k} := e^{\left(C(k,2)^2 q_k^{2N(k,2)}\right)^2},
\end{equation}
we can find $p_k, q_k \in \bn$ with $(p_k,q_k)=1$ such that 
$$
\left| \a - \frac{p_k}{q_k} \right| < \frac{\eps}{q_k^N e^{(C q_k^N)^2}} \leq  \frac{1}{2^{k+1}m_{k}\, C(1)|\!|\!|H_k|\!|\!|_2^2}.
$$
Now, as $\{q_k\mid k\in\bn\}$ is a strictly increasing sequence, also $\{m_k\mid k\in\bn\}$ is strictly increasing, hence $m_{k} \to+\infty$ when $k \to+\infty$. Thus, we can define the following two-sided sequence
\begin{equation}
\label{subs}
k_{n} := 
\left\{
\begin{array}{lr}
1 & \mbox{for } |n| < m_2, \\
\sup \{k\mid m_{k} <|n| < m_{k+1} \} & \mbox{for } |n|\geq m_2.
\end{array}
\right. 
\end{equation}
Notice that $\{k_n\mid n\in\bz\}$ is a symmetric sequence and $\{k_n\mid n\geq0\}$ is non decreasing with $\sup_{|n|} k_n = +\infty$ (hence $k_{n} \to +\infty$ as $|n| \to +\infty$). Then for $|n|\geq m_2$,
$$
C(k_{n},2)^2 q_{k_{n}}^{2N(k_{n},2)} < \sqrt{\ln |n|} < C(k_{n}+1,2)^2 q_{k_{n}+1}^{2N(k_{n}+1,2)}.
$$
Thus, we get for any $n\in\bz$, $|n|\geq m_2$,
\begin{equation}
\label{gs}
 \|D\mathpzc{f}_{k_{n}}^n\| \leq \sqrt{\ln |n|}+1.
\end{equation}
By adapting part of Lemma 5.7 in \cite{FS} to the present case, and using the sequence $m_k$ given by \eqref{seq}, we see that, if
$$
\left| \a - \frac{p_k}{q_k} \right| < \frac{1}{2^{k+1}m_{k}\, C(1)|\!|\!|H_k|\!|\!|_2^2},
$$
then we have $d_1(\mathpzc{f}^n,\mathpzc{f}^n_k) \leq \frac{1}{2^k}$ for any $|n| \leq m_{k+1}$, $k\in\bn$. Indeed, again by Lemma 2.3 in \cite{M} (see also Lemma 5.6 in \cite{FS}), 
$$
 d_1\big(H R_{\a} H^{-1},H R_{\b} H^{-1}\big) \leq C(1) |\!|\!|H|\!|\!|_{2}^{2}|\a-\b|,
$$
and thus
\begin{align*}
 d_1(\mathpzc{f}_{k-1}^n,\mathpzc{f}_k^n) = &d_1\big(H_k R_{n\a_k} H_k^{-1},H_k R_{n\a_{k+1}} H_k^{-1}\big)\\
 \leq &C(1) |\!|\!|H_k|\!|\!|_2^2\,|n||\a_k-\a_{k+1}| \\
 \leq &2C(1) |\!|\!|H_k|\!|\!|_2^2\,|n| |\a-\a_k|\leq \frac{1}{2^k} 
\end{align*}
for any $|n| \leq m_{k}$. Then
$$
 d_1(\mathpzc{f}^n,\mathpzc{f}_{k-1}^n) \leq \sum_{i=k}^{+\infty}d_1(\mathpzc{f}_{i-1}^n,\mathpzc{f}_i^n) \leq \frac{1}{2^{k-1}},
$$
whence the thesis.

Finally, proceeding as in point (vi) above, we get for any $|n| \leq m_{k+1}$,
$$
\|D\mathpzc{f}^n-D\mathpzc{f}^n_k\|\leq \sum_{i=k+1}^{+\infty} d_1(\mathpzc{f}^n_i,\mathpzc{f}^n_{i-1})=\frac{1}{2^{k}},
$$
and thus, by the triangle inequality,
\begin{equation*}
\|D\mathpzc{f}^n\|\leq\frac{1}{2^{k}}+\|D\mathpzc{f}^n_k\|,\quad |n| \leq m_{k+1},\,\,k\in\bn.
\end{equation*}
Now, take a sequence $\{m_{k}\mid k\in\bn\}$ as in \eqref{seq} and a sequence $\{k_{n}\mid n\in\bz\}$ as in \eqref{subs}, and notice that $|n| \to+\infty$ implies $k_{n} \to +\infty$ and $m_{k_{n}} \to +\infty$, the claim follows by \eqref{gs} .
\end{proof}
To conclude this section, we show how to construct examples of such ``ultra-Liouville'' numbers {\bf(UL)} by following the lines in \cite{S}.
Indeed, let $\alpha$ denote the sum of reciprocals of tower powers which are defined as follows:
$$
\a = \sum_{n=1}^{+\infty} \frac{1}{q_n},
$$
where
$$
q_1=2, \quad q_n=(2^n)^{2^{q_{n-1}}}, \quad n=1,2,\dots\,\,.
$$
As the denominator of the $n$-th partial sum $s_n$ of the series is $q_n$, we see that \eqref{uLn} holds true with $s_n=:\frac{p_n}{q_n}$ and for all $n$ satisfying
$$
n 2^{q_n} \log_{\l} 2 \geq q_n^N 
$$
(notice that the previous inequality \eqref{uLn} is eventually satisfied). Indeed, we have
\begin{eqnarray*}
 & & \left| \a - \frac{p_n}{q_n} \right| = \sum_{k=n+1}^{+\infty} \frac{1}{q_k} \leq \frac{1}{(2^{n+1})^{2^{q_n}}} \sum_{k=0}^{+\infty} \frac{1}{2^k} < \frac{1}{(2^n)^{2^{q_n}}} . \\
\end{eqnarray*}

\section{A class of states}
\label{secstat}

\noindent
Let $\m\in\cs(C(\bt))$ be a probability measure on $\bt$. As shown in Proposition 2.1 of \cite{A}, 
\begin{equation}
\label{ommu}
\om_\m(W(f)):=\sum_{m\in\bz}\widecheck{\m}(m)f(m,0)
\end{equation}
is well defined, positive and normalised, and therefore it defines a state on $\ba_{2\a}$. 

The following two results concern the faithfulness properties of the state $\om_\m$.
\begin{prop}
\label{st}
The state $\om_\m\in\cs(\ba_{2\a})$ is faithful if and only if $\supp(\m)=\bt$.
\end{prop}
\begin{proof}
Suppose $\supp(\m)\subsetneq\bt$. Choose a continuous non zero function $\check g\in C(\bt)$ with $\supp(\check g)\cap\supp(\m)=\emptyset$. The function $\check g$ comes from $g\in C^*(\bz)$. By setting
$G(m,n):=g(m)\d_{n,0}$, we get 
\begin{align*}
\om_\m(W(G)^*W(G))=&\om_\m(G^\star *_{2\a} G)=\int_\bt\di\m(z)\sum_{m\in\bz}z^m(g^\star *g)(m)\\
=&\int_\bt\di\m(z)\left(\widecheck{g^\star *g}\right)(z)=\int_\bt\di\m(z)\big|\check g(z)\big|^2=0.
\end{align*}
Thus, $\om_\m$ is not faithful. 

Concerning the reverse implication, we compute
\begin{equation}
\label{1}
\begin{split}
&\om_\m(W(f)^*W(f))=\sum_{m\in\bz}\widecheck{\m}(m)\big(f^\star*_{2\a} f\big)(m,0)\\
=&\sum_{n\in\bz}\int_\bt\di \m(z)\sum_{m\in\bz}\big(ze^{-2\pi\imath\a n}\big)^{m}\left(f^{(-n)\star}* f^{(-n)}\right)(m)\\
=&\sum_{n\in\bz}\int_\bt\di \m(z)\left|\widecheck{f^{(-n)}}\big(ze^{-2\pi\imath\a n}\big)\right|^2
=\sum_{n\in\bz}\int_\bt\di \m_{n}(z)\left|\widecheck{f^{(n)}}(z)\right|^2,
\end{split}
\end{equation}
where $\m_{n}=\m\circ R^{-n}_\a$ with $R^{-n}_\a$ the rotation of the angle $-2\pi n\a$.

If $\om_\m(W(f)^*W(f))=0$, then $\int\di \m_{n}(z)\left|\widecheck{f^{(n)}}(z)\right|^2 = 0$ for each $n\in\bz$. As $\left\{\widecheck{f^{(n)}}\mid n\in\bz\right\}\subset C(\bt)$, this implies $\widecheck{f^{(n)}}=0$, $n\in\bz$, identically on $\bt$ since $\supp(\m_n)=\bt$.
Then for each $n\in\bz$, $f^{(n)}=0$ identically on $\bz$, which means $f(m,n)=0$, for each $n,m\in\bz$. Thus, we conclude that $W(f)=0$.
\end{proof}
Since $\ba_{2\a}$ is simple, we now check the stronger property of the vector state extension $\langle\,{\bf\cdot}\,\xi_\om,\xi_\om\rangle$ of $\om_\m$ to $\pi_{\om_\m}(\ba_{2\a})''$ to be faithful.
\begin{prop}
\label{st1}
Let $R$ be the rotation of the angle $2\pi\a$ on the unit circle $\bt$, and suppose that $\m\circ R^{2n}\preceq\m$, $n\in\bz$. Then the support of $\om_\m$ in the bidual is central: $s(\om_\m)\in Z(\ba_{2\a}^{**})$.
\end{prop}
\begin{proof}
By taking into account Proposition \ref{ccbl}, we apply Theorem 3 in \cite{T}, which reads as follows: given a sequence $\{A_k\mid k\in\bn\}\subset\ba_{2\a}$, if $\om_\m(A_k^*A_k)\to0$ and $\om_\m\big((A_l-A_k)(A_l-A_k)^*\big)\to0$ imply that also $\om_\m(A_kA_k^*)\to0$, then $s(\om_\m)$ is central.

Since $\m_{2n}\preceq\m$ for each $n\in\bz$, we can directly assume that $\m_{2n}\sim\m$, $n\in\bz$. Fix a sequence $\{A_k\mid k\in\bn\}\subset\ba_{2\a}$, with the corresponding sequence of Fourier coefficients $\{f_k\mid k\in\bn\}$ given in \eqref{zfua}, such that $A_k=W(f_k)$, $k\in\bn$. By the same computations as in \eqref{1} for the sequences $f_k$ and $f^*_k$, we obtain respectively,
\begin{align*}
&\om_\m(A_k^*A_k)=\sum_{n\in\bz}\int_\bt\di \m_{n}(z)\left|\widecheck{f_k^{(n)}}(z)\right|^2,\\
&\om_\m(A_kA_k^*)=\sum_{n\in\bz}\int_\bt\di \m_{-n}(z)\left|\widecheck{f_k^{(n)}}(z)\right|^2.
\end{align*}
By defining
$$
g_k^{(n)}(z):=\widecheck{f_k^{(n)}}\left(ze^{2\pi\imath\a n}\right),
$$
we obtain
\begin{equation}
\label{3a}
\om_\m(A_k^*A_k)=\sum_{n\in\bz}\int_\bt\di \m(z)\left|g_k^{(n)}(z)\right|^2,
\end{equation}
\begin{equation}
\label{3b}
\om_\m(A_kA_k^*)=\sum_{n\in\bz}\int_\bt\di \m_{-2n}(z)\left|{g_k^{(n)}}(z)\right|^2.
\end{equation}
Moreover, we have
\begin{equation}
\label{3c}
\om_\m\big((A_l-A_k)(A_l-A_k)^*\big)=\sum_{n\in\bz}\int_\bt\di\m_{-2n}(z)\left|g_l^{(n)}(z)-g_k^{(n)}(z)\right|^2.
\end{equation}
Suppose now that $\om_\m(A_k^*A_k)\to0$. By \eqref{3a}, this implies that $\big|g_k^{(n)}\big|^2\to0$ in $\m$-measure for each $n\in\bz$. By assumption, $\di\m_{-2n}(z)=d_n(z)\di\m(z)$ for a $L^1$-function $d_n$ with $1/d_n$ again in $L^1$. This implies that $g_k^{(n)}\to0$ in $\m_{-2n}$-measure as well. In fact, let $\eps>0$ be given, and define 
$$
E^{(n,\eps)}_k:=\left\{z\in\bt\mid \big|g_k^{(n)}(z)\big|^2\geq\eps\right\}.
$$ 
Notice that $0\leq\chi_{E^{(n,\eps)}_k}\leq1$, and $\chi_{E^{(n,\eps)}_k}\to0$ in $L^1(\bt,\m)$ as $k \to+\infty$.
By the Lebesgue Dominated Convergence theorem, we then get for each $n \in \bz$,
$$
\int_{E^{(n,\eps)}_k}\di\m_{-2n}(z)=\int_\bt\chi_{E^{(n,\eps)}_k}(z)d_n(z)\di\m(z)\to0.
$$
By \eqref{3c}, $\om_\m\big((A_l-A_k)(A_l-A_k)^*\big)\to0$ as $k,l \to +\infty$, implies that $\left\{\oplus_{n\in\bz}g_k^{(n)}\right\}$ is a Cauchy sequence in $\bigoplus_{n\in\bz}L^2(\bt,\m_{-2n})$, hence it converges to a function $g=\oplus_{n\in\bz}g^{(n)}$. 
Thus, on one hand $g_k^{(n)}\to0$ in $\m_{-2n}$-measure, and on the other hand $\oplus_{n\in\bz}g_k^{(n)}\to g$ in $\bigoplus_{n\in\bz}L^2(\bt,\m_{-2n})$. But this necessarily implies that $g=0$. Therefore,
$$
\sum_{n\in\bz}\int_\bt\big|g_k^{(n)}\big|^2\di\m_{-2n}\to0,
$$ 
and thus by \eqref{3b} we get $\om_\m(A_kA_k^*)\to0$.
\end{proof}
We conclude the present section by exhibiting other states associated with measures supported on the transverse circle of $\bt^2$. Indeed, fix any probability measure
$\n$ on $\bt$ and consider the
measure on $\bt^2$ uniquely defined by the Riesz-Markov Representation theorem as
$$
\m_\n(F)\equiv\int_{\bt^2}F(z,w)\di\m_\n(z,w):=\int_{\bt^2}F(zw,w)\di\n(z)\di m(w),\,\, F\in C(\bt^2).
$$
Concerning its characteristic function, we compute
\begin{align*}
\widecheck{\m_\n}(m,n)=&\int_{\bt^2}z^mw^{n}\di\m_\n(z,w)=\int_{\bt^2}z^mw^{m+n}\di\n(z)\di m(w)\\
=&\int_{\bt}\di\n(z)z^m\int_{\bt}\di m(w)w^{m+n}=\widecheck{\n}(m)\d_{m+n,0}.
\end{align*}
For $W(f)\in\ba_{2\a}$ with $f\in\cb(\bz^2)$ with finite support, we set
\begin{equation}
\label{trczm}
\om(W(f)):=\sum_{m,n}e^{2\pi\imath\a m^2}\widecheck{\m_\n}(m,n)f(m,n).
\end{equation}
The above function would define a state on $\ba_{2\a}$. In fact, we have
\begin{prop}
The function $\om$ in \eqref{trczm} is positive on the unital $*$-subalgebra of $\ba_{2\a}$ given by $\{W(f)\in\ba_{2\a}\mid f\in\cb(\bz^2)\,{\rm with\,\, finite\,\, support}\}$.
\end{prop}
\begin{proof}
For a finitely supported functions $f\in\cb(\bz^2)$, put $g(m,n):=f(m,n)e^{-2\pi\imath\a mn}$. We compute
\begin{align*}
&\om(W(f)^*W(f))=\om(W(f^\star*_{2\a} f))=\sum_{m,n}(g^\star*g)(m,n)\widecheck{\m_\n}(m,n)\\
=&\int_{\bt^2}\left(\widecheck{g^\star*g}\right)(z,w) \di\m_\n(z,w)
=\int_{\bt^2}|\check g(z,w)|^2\di\m_\n(z,w)\geq0.
\end{align*}
\end{proof}

\section{The GNS representation and the modular structure}
\label{gnsmosc}

\noindent
We now proceed to the description of the GNS representation and the associated modular structure, provided that $s(\om_\m)\in Z(\ba_{2\a}^{**})$, relative to the states $\om_\m$ described in the previous section.
\begin{prop}
\label{ggnnss}
The GNS representation $\big(\ch_{\om_\m},\pi_{\om_\m},\xi_{\om_\m}\big)$ associated with the state $\om_\m$ in \eqref{ommu} is given by
\begin{equation}
\label{ggnnss5}
\begin{split}
\ch_{\om_\m}&=\bigoplus_{n\in\bz}L^2(\bt,\m_n),\\
(\pi_{\om_\m}(W(f))g)_n(z)&=\sum_{l\in\bz}\left(\widecheck{f^{(l)}}\circ R^{n-l}\right)(z)\left(g_{n-l}\circ R^{-l}\right)(z),\\
(\xi_{\om_\m})_n(z)&=\d_{n,0},\quad z\in\bt,\,\,n\in\bz.
\end{split}
\end{equation}
\end{prop}
\begin{proof}
We proceed first by constructing a representation of $\ba_{2\a}$, and then by showing that it is indeed the GNS representation associated with the state $\om_\m$. To this aim, we start with the computations involving the generators $f(m,n)=\d_{k,m}\d_{l,n}$. In this case, by \eqref{ccrba1} we have
\begin{align*}
(f*_{2\a} g)(m,n)&=\sum_{(p,q)\in\bz^2}\d_{k,p}\d_{l,q}g(m-p,n-q)e^{-2\pi\imath\a(mq-np)}\\
&=g(m-k,n-l)e^{-2\pi\imath\a(ml-nk)},
\end{align*}
and thus
\begin{align*}
\widecheck{(f*_{2\a} g)^{(n)}}(z) &=\sum_{m\in\bz} g(m-k,n-l)e^{-2\pi\imath\a(ml-nk)} z^m\\
&=\sum_{m\in\bz} g^{(n-l)}(m-k)e^{-2\pi\imath\a[(m-k)l-(n-l)k]} z^{m-k} z^k\\
&=\left(\sum_{m\in\bz} g^{(n-l)}(m-k) \big(R^{-l}(z)\big)^{m-k}\right) \big(R^{n-l}(z)\big)^k\\
&=\big(R^{n-l}(z)\big)^{k}\widecheck{g^{(n-l)}}\circ R^{-l}(z).
\end{align*}
By recalling that $\{W(f)\mid f\in\cb(\bz^2)\,\text{with finite support}\}$ is a dense $*$-algebra of $\ba_{2\a}$, for a finitely supported $f(m,n)=\sum_{k,l}f(k,l)\d_{k,m}\d_{l,n}$ this leads to
\begin{equation}
\label{vide}
\begin{split}
\widecheck{(f*_{2\a} g)^{(n)}}(z)
=&\sum_{l\in\bz}\bigg(\sum_{k\in\bz}f(k,l)\big(R^{n-l}(z)\big)^{k}\bigg)\widecheck{g^{(n-l)}}\big(R^{-l}(z)\big)\\
=&\sum_{l\in\bz}\widecheck{f^{(l)}}\big(R^{n-l}(z)\big)\widecheck{g^{(n-l)}}\big(R^{-l}(z)\big).
\end{split}
\end{equation}
By performing the same calculations as in \eqref{1}, we get for any finitely supported $f,g,h\in \cb(\bz^2)$,
$$
\om_\m\big(W(h)^*W(f)W(g)\big)=\sum_{n\in\bz}\int_\bt \widecheck{(f*_{2\a} g)^{(n)}}(z)\overline{\widecheck{h^{(n)}}(z)}\di\m_n(z),
$$
which, together with \eqref{vide}, leads to
$$
\om_\m\big(W(h)^*W(f)W(g)\big)=\sum_{n\in\bz}\int_\bt\di\m_n\bigg(\sum_{l\in\bz}\widecheck{f^{(l)}}\circ R^{n-l}(z)g_{n-l}\circ R^{-l}(z)\bigg)\overline{\widecheck{h^{(n)}}(z)}.
$$
The computation in \eqref{vide} also yields the following two facts. First, the cyclicity of $\xi_{\om_\m}$:
$$
(\pi_{\om_\m}(W(f))\xi_{\om_\m})_n(z)=\sum_{l\in\bz}\widecheck{f^{(l)}}\big(R^{n-l}(z)\big)\d_{n,l}=\widecheck{f^{(n)}}(z),
$$
that is $\{\pi_{\om_\m}(W(f))\xi_{\om_\m}\mid f\in\cb(\bz^2)\,\text{with finite support}\}$ is dense in the Hilbert space $\bigoplus_{n\in\bz}L^2(\bt,\m_n)$.
Second, the vector state given by $\xi_{\om_\m}$ is precisely $\om_\m$:
\begin{align*}
\big\langle\pi_{\om_\m}(W&(f))\xi_{\om_\m},\xi_{\om_\m}\big\rangle
=\sum_{n\in\bz}\int_\bt\widecheck{f^{(n)}}(z)\d_{n,0}\di\m_n(z)
=\int_\bt\widecheck{f^{(0)}}(z)\di\m(z)\\
=&\int_\bt\di\m(z)\sum_{m\in\bz}f(m,0)z^m
=\sum_{m\in\bz}f(m,0)\int_\bt z^m\di\m(z)\\
=&\sum_{m\in\bz}f(m,0)\widecheck{\m}(m)
=\om_\m\big(W(f)\big).
\end{align*}
We now set $\ch_{\om_\m}=\bigoplus_{n\in\bz}L^2(\bt,\m_n)$, and denote by $A_{\om_\m}$ the element $A\in\ba_{2\a}$ viewed as a vector in $\ch_{\om_\m}$, {\it i.e.} we identify $W(f)_{\om_\m}$ with $\bigoplus_{n\in\bz} \widecheck{f^{(n)}} \in \ch_{\om_\m}$. 

Finally, we look at the multilinear form on $\ch_{\om_\m}$ given by
$$
\big\langle\pi_{\om_\m}(W(f))W(g)_{\om_\m},W(h)_{\om_\m}\big\rangle=\om_\m\big(W(h)^*W(f)W(g)\big).
$$
The statement then follows by uniqueness, up to unitary equivalence, of the GNS representation.
\end{proof}
Suppose that $\m$ satisfies the conditions of Proposition \ref{st1}. This means that $\m_{-n}\sim\m_{n}$, $n\in\bz$. By setting
$$
\di\m_{-n}(z)=h_n(z)\di\m_{n}(z)
$$
for the corresponding Radon-Nikodym derivative, we then have 
$$
h_{-n}(z)=h^{-1}_{n}(z),\quad n\in\bz,
$$
$\m_n$-almost everywhere.

In this case, the embedding of $\ba_{2\a}$ in $\ch_{\om_\m}$ is faithful. Moreover, since $W(f_k)_{\om_\m} \to 0$ implies $\bigoplus_{n\in\bz}\widecheck{f^{(n)}_k} \to 0$ in $\ch_{\om_\m}$ (see the proof of Proposition \ref{st1}), it easily follows that, on elements $x\in\ch_{\om_\m}$ of the form $x_n(z)=\widecheck{f^{(n)}}(z)$, the map $S_o:W(f)\to W(f^\star)$ given by
$$
(S_ox)_n(z)=\overline{x_{-n}(z)},
$$
is closable. Denoting its closure by $S$, we have $S=J\D^{1/2}$.
\begin{thm}
\label{mod}
Suppose that for $\m\in\cs(C(\bt))$, $\m_{-n}\sim\m_{n}$ with Radon-Nikodym derivative $\frac{\di\m_{-n}}{\di\m_{n}}=h_n$,
$n\in\bz$. Then we have for the corresponding Tomita modular operator $\D$ and conjugation $J$, 
$$
\cd_\D=\bigg\{x\in\ch_{\om_\m} \mid\sum_{n\in\bz}\int_\bt|h_n(z)x_{n}(z)|^2\di\m_n(z)<+\infty\bigg\},
$$ 
and
\begin{equation*}
(\D x)_n(z)=h_n(z)x_{n}(z), \quad(Jx)_n(z)=h_n^{1/2}(z)\overline{x_{-n}(z)}.
\end{equation*}
\end{thm}
\begin{proof}
We start by determining $S^*$. On elements of the form $x_n=\widecheck{f^{(n)}}$, $y_n=\widecheck{g^{(n)}}$ for $W(f), W(g)\in\ba_{2\a}$, we compute
\begin{align*}
\langle S^*x,y\rangle=&\langle S_oy,x\rangle=\sum_{n\in\bz}\int_\bt\overline{x_n(z)y_{-n}(z)}\di\m_n(z)\\
=&\sum_{n\in\bz}\int_\bt\overline{y_n(z)x_{-n}(z)}\di\m_{-n}(z)\\
=&\sum_{n\in\bz}\int_\bt h_n(z)\overline{x_{-n}(z)y_n(z)}\di\m_{n}(z)\\
=&\sum_{n\in\bz}\int_\bt h_n(z)(S_ox)_n(z)\overline{y_n(z)}\di\m_{n}(z).
\end{align*}
This leads to
$$
(S^*)_n=M_{h_n}S_n,\quad n\in\bz
$$
which, combined with $S^2\subset I$ ({\it cf.} 10.1 in \cite{SZ}), yields
$$
\D_n=M_{h_n},\quad n\in\bz.
$$
Concerning the conjugation $J$, by taking into account that $J\supset S\D^{-1/2}$ we easily compute for elements $x\in\ch_{\om_\m}$ as above,
$$
(J x)_n(z)=\big(S(\D^{-1/2}x)\big)_n(z)=h_{-n}^{-1/2}(z)\overline{x_{-n}(z)}=h_{n}^{1/2}(z)\overline{x_{-n}(z)}.
$$
\end{proof}
Concerning the spectrum of the modular operator, denote by $\car_{\rm ess}(h_n)$ the essential range of the measurable function $h_n$ w.r.t. the measure $\m_n$. By Proposition VIII.3.1 in \cite{RS}, we easily get
$$
\s(\D)=\overline{\bigcup_{n\in\bn}\left(\car_{\rm ess}(h_n)\bigcup1/\car_{\rm ess}(h_{n})\right)}.
$$
As a simple application that covers also some examples dealt with in \cite{CM}, we consider the particular case corresponding to inner perturbations of the trace. In particular, set $\f:=\t(W(k)\,{\bf\cdot}\,)$ with $k(m,n)=K(m)\d_{n,0}$. Notice that, in the language of Section \ref{secstat}, $\t=\om_{\m^\t}$, where for each $n\in\bz$, $\m_n^\t=m$, with $m$ the normalised Lebesgue measure on the unit circle. A simple application of \eqref{ggnnss5} to
$$
\f(W(h)^*W(f))=\t((W(h)^*W(f)W(k)),
$$
with $g_n=W(k)_\t\equiv\pi_\t(W(k))$, yields $\f=\om_{\m^\f}$ with the corresponding measures $\m^\f_n$ given by 
$$
\di\m^\f_n=\widecheck{K}\circ R^{-n}\di m,\quad n\in\bz.
$$

\section{Type $\ty{III}$ representations}
\label{t3rp}

\noindent
In order to find out type $\ty{III}$ representations, we restrict the matter to Liouville numbers. 
Thus, we fix an irrational number $\a\in(0,1/2)$, tacitly assuming that it will be a Liouville one. Our analysis also applies to the diophantine case, providing type $\ty{II_1}$ representations which are then equivalent to that of the trace.

With $R$ the rotation of the angle $2\pi\a$, we consider an orientation preserving diffeomorphism $\mathpzc{f}$ with $\r(\mathpzc{f})=2\a$. As $\mathpzc{f}=\mathpzc{h}_\mathpzc{f}\circ R^2\circ\mathpzc{h}_\mathpzc{f}^{-1}$ for a unique homeomorphism $\mathpzc{h}_\mathpzc{f}$ such that $\mathpzc{h}_\mathpzc{f}(1)=1$, we define $T:=\mathpzc{h}_\mathpzc{f}\circ R\circ \mathpzc{h}_\mathpzc{f}^{-1}$. As explained in Lemma \ref{equ}, $T^2=\mathpzc{f}$, and we can set $\mathpzc{h}_\mathpzc{f}\equiv \mathpzc{h}_{T^2}$. Notice that, contrarily to its square $\mathpzc{f}$, $T$ is merely a non smooth homeomorphism for all the non type 
$\ty{II_1}$ cases treated in the present paper. 

We also consider the measure $\m:=m\circ\mathpzc{h}_{T^2}$.
Concerning the sequence of measures $\m_n=\m\circ R^{-n}$, we have the following
\begin{lem}
\label{qihk}
For each $n\in\bz$, we have $\m_{-n}\sim\m_{n}$ with
$$
\frac{\di\m_{-n}}{\di\m_{n}}(z)=\frac{T^{-n}(\mathpzc{h}_{T^2}(z))(DT^{2n})(T^{-n}(\mathpzc{h}_{T^2}(z)))}{T^{n}(\mathpzc{h}_{T^2}(z))},\quad z\in\bt.
$$
\end{lem}
\begin{proof}
Since we know that $m\circ T^{n}\sim m\circ T^{-n}$ for each $n\in\bz$, we have
\begin{align*}
\m_{-n}\equiv&\m\circ R^n=m\circ \mathpzc{h}_{T^2}\circ R^n=m\circ T^n\circ \mathpzc{h}_{T^2}\\
\sim&m\circ T^{-n}\circ \mathpzc{h}_{T^2}\equiv m_n\circ \mathpzc{h}_{T^2}=\m_n,
\end{align*}
hence $\m_{-n}\sim\m_{n}$.

Concerning the corresponding Radon-Nikodym derivative, by \eqref{rcanz} we have for each $n\in\bz$ and $z\in\bt$,
$$
\frac{\di m_{-n}}{\di m_{n}}(z)=\frac{\di m_{-2n}}{\di m}\big(T^{-n}(z)\big)=\frac{T^{-n}(z)(DT^{2n})(T^{-n}(z))}{T^{n}(z)}.
$$
We then compute
\begin{align*}
\di(\m\circ R^n)=&\di(m\circ \mathpzc{h}_{T^2}\circ R^n)=\di(m\circ T^n\circ \mathpzc{h}_{T^2})\\
=&\frac{\di m_{-n}}{\di m_{n}}\circ \mathpzc{h}_{T^2}\,\di(m\circ T^{-n}\circ \mathpzc{h}_{T^2})\\
=&\frac{\di m_{-n}}{\di m_{n}}\circ \mathpzc{h}_{T^2}\,\di(m\circ \mathpzc{h}_{T^2}\circ R^{-n})\\
=&\frac{\di m_{-n}}{\di m_{n}}\circ \mathpzc{h}_{T^2}\,\di(\m\circ R^{-n}),
\end{align*}
and the thesis follows.
\end{proof}

In order to determine the type of the representation $\pi_{\om_\m}$, we use another representation, unitarily equivalent to the GNS one in Proposition \ref{ggnnss}, which is more convenient for our purpose. Indeed, for the Hilbert space
$$
\ck:=\ell^2\big(\bz; L^2(\bt,\m)\big)\cong\bigoplus_{n\in\bz}L^2(\bt,\m),
$$
we define the operator $V:\ch_{\om_\m}\to\ck$ as
$$
(Vg)_n(z):=(g_n\circ R^{n})(z),\quad n\in\bz,\,\,z\in\bt.
$$
It is immediate to see that $V$ is unitary with inverse $V^*:\ck\to\ch_{\om_\m}$ given by
$$
(V^*g)_n(z)=(g_n\circ R^{-n})(z),\quad n\in\bz,\,\,z\in\bt.
$$
We can check that, after performing the above unitary equivalence in \eqref{ggnnss5}, the GNS representation can be also written as
\begin{equation}
\label{ggnnss565}
\begin{split}
\ch_{\om_\m}&=\ell^2\big(\bz;L^2(\bt,\m)\big),\\
(\pi_{\om_\m}(W(f))g)_n(z)&=\sum_{l\in\bz}\left(\widecheck{f^{(l)}}\circ R^{2n-l}\right)(z)g_{n-l}(z),\\
(\xi_{\om_\m})_n(z)&=\d_{n,0},\quad z\in\bt,\,\,n\in\bz.
\end{split}
\end{equation}
The modular operator and the modular conjugation associated with the state $\om_\m$, necessarily of central support ({\it cf.} Proposition \ref{st1}), such as that described in Theorem \ref{mod}, are then given by
\begin{equation}
\label{ggnnss765}
\begin{split}
(\D x)_n(z)=&(h_n\circ R^n)(z)x_{n}(z),\\
(Jx)_n(z)=&(h_n\circ R^n)^{1/2}(z)\overline{(x_{-n}\circ R^{2n})(z)},
\end{split}
\end{equation}
for each $n\in\bz$ and $z\in\bt$.

We now check that $\pi_{\om_\m}(\ba_{2\a})''$ is indeed isomorphic to a crossed product. 
\begin{prop}
\label{crprac}
With the above notations, we have 
$$
\pi_{\om_\m}(\ba_{2\a})''\sim L^\infty(\bt,\m)\ltimes_{R_{2\a}}\bz.
$$
\end{prop}
\begin{proof}
Consider
\begin{equation}
\label{casca}
\ca:=\big\{W(f)\mid f(m,n)=F(m)\d_{n,0},\, F\in\cb(\bz)\,{\rm with\,\,finite\,\,support}\big\}.
\end{equation}
Notice that $\ca$ is a $*$-algebra which can be viewed as a $*$-algebra of continuous functions, norm dense in the copy of $C(\bt)$ canonically embedded in $\ba_{2\a}$, and weakly dense in 
$\pi_{\om_\m}(\ca)''\sim L^\infty(\bt,\m)$. For $f(m,n)=F(m)\d_{n,0}$ as above, $\widecheck{f^{(n)}}=H\d_{n,0}$ where $H:=\widecheck{F}\in C(\bt)$. 

Consider now the natural action $\r$ on functions $f\in L^\infty(\bt,\m)$, dual of the rotation of $4\pi\a$, given by
$$
\r(f)(z):=f\circ R^{-2}(z),\quad z\in\bz.
$$
By \eqref{ggnnss565}, it is immediate to show that $C(\bt) \subset L^\infty(\bt,\m)$ are both naturally embedded in $\pi_{\om_\m}(\ba_{2\a})''$ as the multiplication operator by the functions
$$
\pi(H)_n:=M_{H\circ R^{2n}},\quad n\in\bz.
$$
Concerning the action of $\bz$, we pick $h_k(m,n):=\d_{m,0}\d_{k,n}$ and put
$$
\l_k:=\pi_{\om_\m}(W(h_k))\,, \quad k\in\bz.
$$
Such unitary operators act on $\ell^2\big(\bz;L^2(\bt,\m)\big)$ as the $k$-step shift:
$$
(\l_kg)_n=g_{n-k},\quad k,n\in\bz.
$$
Obviously, $\l_k=\l^k$, $k\in\bz$, and $\l^{-1}=\l^*$, where $\l$ is the one-step shift
\begin{equation}
\label{lash}
(\l g)_n=g_{n-1},\quad n\in\bz.
\end{equation}
Then $\pi(\ca)$ and $\{\l_k\mid k\in\bz\}$ generate $\ba_{2\a}$ as a $C^*$-algebra, and $\pi_\om(\ba_{2\a})''$ as a von Neumann algebra. In addition, since the crossed product condition (see {\it e.g.} Section 2.7.1 of \cite{BR})
$$
\l_k\pi(H)\l_k^*=\pi(\r^{k}(H)),\quad H\in L^\infty(\bt,\m),\,\,k\in\bz
$$
is easily verified, the result follows.
\end{proof}
 Since $\ba_{2\a}$ is a simple $C^*$-algebra, we recover the representation-independent well known fact that $\ba_{2\a}\sim C(\bt)\ltimes_{R_{2\a}}\bz$, where we have denoted with an abuse of notation, by $R_{2\a}$  the action $\b$ on functions corresponding to the rotation by the angle $4\pi\a$.

The main properties of the representation $\pi_{\om_\m}$ are then described in the following
\begin{thm}
\label{thmcz}
For $\a\in(0,1/2)\backslash\bq$, let $\mathpzc{f}$ be an orientation preserving diffeomorphism of $\bt$ satisfying {\rm{\bf(A)}} in Section \ref{sec2} w.r.t. the rotation $R_{2\a}$ by the angle $4\pi\a$. Consider the measure $\m=m\circ \mathpzc{h}_\mathpzc{f}$, together with the state $\om_\m\in\cs(\ba_{2\a})$ given in \eqref{ommu}. 

Then $\pi_{\om_\m}(\ba_{2\a})''$ is a hyperfinite factor acting in standard form on $\ch_{\om_\m}$, whose type \emph{(}necessarily $\ty{II_1}$, $\ty{II_\infty}$, or $\ty{III_\l}$, $\l\in[0,1]$\emph{)} is determined by the ratio set $r([\m],R_{2\a})$ \emph{(}or equivalently by $r([m],\mathpzc{f})$\emph{)}, provided it is not of type $\ty{II_1}$. Furthermore, for $\ca$ given in \eqref {casca}, $\pi_{\om_\m}(\ca)''$ is maximal abelian in $\pi_{\om_\m}(\ba_{2\a})''$.
\end{thm}
\begin{proof}
We already noticed that the Lebesgue measure $m$ is quasi invariant and ergodic for the natural action of $\mathpzc{f}$ on $\bt$. Consequently, $\m$ is also quasi invariant and ergodic for the rotation $R^2$ by the angle $4\pi\a$ on $\bt$. By Lemma \ref{qihk}, $\m$ satisfies the condition of Proposition \ref{st1}, and thus $\pi_{\om_\m}(\ba_{2\a})''$ acts in standard form on 
$\ch_{\om_\m}$. 

In Proposition \ref{crprac}, we have proven that $\pi_{\om_\m}(\ba_{2\a})''\sim L^\infty(\bt,\m)\ltimes_{R_{2\a}}\bz$. Moreover, as the action of $\bz$ on $L^\infty(\bt,\m)$ generated by $R^2$ is ergodic, it is also free (see {\it e.g.} pag. 363 in Section V.7 of \cite{T1}). Then by Theorem XIII.1.5 and Corollary XIII.1.6 in \cite{T1} (see also \cite{K1, K3} for the original proofs), $\pi_{\om_\m}(\ba_{2\a})''$ is a factor containing $\pi_{\om_\m}(\ca)''$ as a maximal abelian subalgebra. In addition, it is hyperfinite as $\bz$ is amenable, see {\it e.g.} Theorem 4.4 of \cite{K2a} or also \cite{CFW}.

Concerning the type, trivially $\pi_{\om_\m}(\ba_{2\a})''$ cannot be of type $\ty{I}$ since $\m$ is nonatomic. It is of type $\ty{II_1}$ if and only if the measure class $[\m]$ contains a probability measure which is invariant under the action generated by $R^2$, see Theorem XIII.1.7 in \cite{T1}. In this case, $\m\sim m$ and $\pi_{\om_\m}(\ba_{2\a})\sim\pi_{\t}(\ba_{2\a})$ by uniqueness. This always occurs when $\a$ is diophantine. 

In the remaining cases, if $\pi_{\om_\m}(\ba_{2\a})$ is not of $\ty{II_1}$, its type is determined by the Connes invariant ${\rm S}\big(\pi_{\om_\m}(\ba_{2\a})\big)$. The assertion follows by noticing that the last coincides with 
the Krieger-Araki-Woods ratio set ({\it cf.} \cite{K2}) $r([\m],R^{2})$ associated to the underlying ergodic $W^*$-dynamical system $\big(L^\infty(\bt,\m), R^{2}\big)$, see  {\it e.g.} Section 2.1 in \cite{RMPR}.
\end{proof}
Collecting together the previous theorem and Proposition \ref{catsum}, non type $\ty{II_1}$ representations of the noncommutative 2-torus $\ba_{2\a}$ can be explicitly constructed for any Liouville number $\a$.

\section{The deformed Dirac operator}
\label{dirac}

\noindent
The present section is devoted to define a deformed Dirac operator associated with one of the representations of Theorem \ref{thmcz} whose twist is associated to the modular operator. The reader is referred to \cite{CM}, where such a structure is investigated for the type $\ty{II_1}$ case arising from inner bounded "smooth" perturbations of the canonical trace. As the state $\om_\m$ in \eqref{ommu} is fixed, from now on we denote it simply by $\om$ if is not otherwise specified.

For such a purpose, we work on the space $\bigoplus_{n\in\bz}L^2(\bt,m)$ rather than on $\bigoplus_{n\in\bz}L^2(\bt,\m)$. These Hilbert spaces are unitarily equivalent, such an equivalence being implemented by the unitary operator $u: L^2(\bt,\m)\to L^2(\bt, m)$ as follows:
$$
g\in L^2(\bt,\m)\mapsto ug:=g\circ\mathpzc{h}^{-1}_{T^2}\in L^2(\bt, m).
$$
Under the unitary equivalence realised by the direct sum of infinitely many copies of $u$, it is easy to see that, again with an abuse of notation,
\begin{equation}
\label{ggnnss666}
\begin{split}
\ch_{\om}&=\ell^2\big(\bz;L^2(\bt, m)\big),\\
(\pi_{\om}(W(f))g)_n(z)&=\sum_{l\in\bz}\left(\widecheck{f^{(l)}}\circ \mathpzc{h}^{-1}_{T^2}\circ T^{2n-l}\right)(z)g_{n-l}(z),\\
(\xi_{\om})_n(z)&=\d_{n,0},\quad z\in\bt,\,\,n\in\bz.
\end{split}
\end{equation}
By \eqref{rcanz} and \eqref{ggnnss765}, for $n\in\bz$ and $z\in\bt$, the related modular structure is then given by
\begin{align*}
(\D x)_n(z)=&\frac{z(DT^{2n})(z)}{T^{2n}(z)}x_{n}(z),\\ 
(Jx)_n(z)=&\bigg[\frac{z(DT^{2n})(z)}{T^{2n}(z)}\bigg]^{1/2}\overline{(x_{-n}\circ T^{2n})(z)}.
\end{align*}
We set
\begin{equation}
\label{rndica}
\d_n(z):=\frac{\di m_{-2n}}{\di m}(z)\equiv\frac{z(DT^{2n})(z)}{T^{2n}(z)},\quad z\in\bt,\,\,n\in\bz,
\end{equation}
so that the modular operator is given by $\D=\bigoplus_{n\in\bz}M_{\d_n}$, $M_f$ denoting the multiplication operator by the function $f$.

After changing representation according to \eqref{ggnnss666}, on the Hilbert space
\begin{equation}
\label{accam}
\begin{split}
\ch_\om\oplus\ch_\om\equiv&
\bigg(\bigoplus_\bz L^2(\bt, m)\bigg)\bigoplus\bigg(\bigoplus_\bz L^2(\bt, m)\bigg) \\ 
=&\bigoplus_\bz\bigg(L^2(\bt, m)\bigoplus L^2(\bt, m)\bigg)
\end{split}
\end{equation}
we put
$$
D:=\begin{pmatrix} 
	 0 &L \\
	L^* & 0\\
     \end{pmatrix},
$$
where 
\begin{equation}
\label{eldcz}
L=(\partial_1+\imath\partial_2).
\end{equation}
Here, the $\partial_i$ are the partial derivatives w.r.t. the angles $\th_i$, $i=1,2$, of functions $f\big(e^{\imath\th_1},e^{\imath\th_2}\big)$ on the 2-torus $\bt^2$, which in our context assume the form
$$
(\partial_1g)_n(z):=\imath z\frac{\di g_n}{\di z}(z),\quad (\partial_2g)_n(z):=\imath ng_n(z),\quad z\in\bt,\,\,n\in\bz.
$$
We then compute
$$
D=\bigoplus_{n\in\bz}D_n=\bigoplus_{n\in\bz}
     \begin{pmatrix} 
	 0 &\big(\imath z\frac{\di\,\,}{\di z}-n I\big)\\
	\big(-\imath z\frac{\di\,\,}{\di z}-n I\big)& 0\\
     \end{pmatrix},
$$
where each $D_n= \begin{pmatrix} 
	 0 &L_n\\
	L_n^*& 0\\
     \end{pmatrix}$, with $L_n:=\imath z\frac{\di\,\,}{\di z}-n I$,
acts on the direct sum $L^2(\bt, m)\bigoplus L^2(\bt, m)$ of two copies of $L^2(\bt, m)$.

For the convenience of the reader, we write down the spectral resolution of $D$. By noticing that $\{z^n\mid n\in\bz\}$ is an orthonormal basis of $L^2(\bt,m)$, one gets
\begin{eqnarray*}
 & & \begin{pmatrix} 
	0 &\big( \imath z\frac{\di\,\,}{\di z}-n I\big)\\
	\big(-\imath z\frac{\di\,\,}{\di z}-n I\big)& 0\\
     \end{pmatrix} \, \begin{pmatrix} 
	 a^{(\pm)}_m z^m\\
	b^{(\pm)}_m z^m\\
     \end{pmatrix} = \pm \sqrt{n^2+m^2} \begin{pmatrix} 
	 a^{(\pm)}_m z^m\\
	b^{(\pm)}_m z^m\\
     \end{pmatrix}
\end{eqnarray*}
if and only if $a^{(\pm)}_m = \pm \frac{\imath m-n}{\sqrt{n^2+m^2}} b^{(\pm)}_m$, provided $m,n\neq0$. The eigenspace corresponding to $m,n=0$ has degeneracy 2, and we can choose as eigenfunctions the constant vectors
$$
\eps_{00}^{(\pm)} = \frac1{\sqrt2}\begin{pmatrix} 
 1\\
 \pm1\\
\end{pmatrix}.
$$
The remaining cases provide simple eigenvectors, given after normalisation by
$$
\eps_{mn}^{(\pm)}(z) = \frac1{\sqrt2}\begin{pmatrix} 
 \frac{\imath m-n}{\sqrt{n^2+m^2}} \\
 \pm1\\
\end{pmatrix}z^m,\quad (m,n)\in\bz^2\backslash\{(0,0)\},\,z\in\bt.
$$
Thus, $\s(D)=\big\{\pm\sqrt{m^2+n^2}\mid m,n\in\bn\big\}$. Finally, setting
$$
e_{nm}^{(\pm)}:=\bigoplus_{k\in\bz}\delta_{n,k} \eps_{mn}^{(\pm)} \in \ch_\om\oplus\ch_\om,
$$
we see that the orthonormal system $\big\{e_{nm}^{(\pm)}\mid m,n\in\bz\big\}$ is a basis for $\ch_\om\oplus\ch_\om$ made of eigenvectors of the untwisted Dirac operator $D$.

The deformed Dirac operator with twisting determined by the modular operator, is here defined as follows. Denote $AC(\bt)$ the set of all absolutely continuous complex valued functions on the unit circle.
The Sobolev-Hilbert space $H^1(\bt)$ is given by
$$
H^1(\bt):=\big\{f\in AC(\bt)\mid f'\in L^2(\bt,m)\big\}.
$$ 
For each $n\in\bz$, we put
$$
\cd_{D^\s_n}:=H^1(\bt)\oplus H^1(\bt)\subset L^2(\bt,m)\oplus L^2(\bt,m),
$$
and define the (unbounded) operators
$$
 D^\s_n:=\begin{pmatrix} 
	 0 &M_{\d^{-1}_n}L_n\\
	L_n^*M_{\d^{-1}_n}& 0\\
     \end{pmatrix}.
$$
Then the deformed Dirac operator $D^\s$ is defined as
\begin{equation}
\label{ndsd}
D^\s:=\begin{pmatrix} 
	 0 &\D^{-1}L\\
	L^*\D^{-1}& 0\\
     \end{pmatrix}=\bigoplus_{n\in\bz}
     \begin{pmatrix} 
	 0 &M_{\d^{-1}_n}L_n\\
	L_n^*M_{\d^{-1}_n}& 0\\
     \end{pmatrix}=\bigoplus_{n\in\bz}D^\s_n,
\end{equation}
on the domain 
$$
\cd_{D^\s}:=\bigg\{\xi\in\bigoplus_{n\in\bz}\cd_{D^\s_n}\mid\sum_{n\in\bz}\|D^\s_n\xi_n\|^2<+\infty\bigg\}.
$$
\begin{prop}
\label{dsukz}
The operator $D^\s$ with domain $\cd_{D^\s}$ is selfadjoint.
\end{prop}
\begin{proof}
We start by showing that the operators $D^\s_n$ are selfadjoint on $\cd_{D^\s_n}:=H^1(\bt)\oplus H^1(\bt)$. First, we notice that $M_f H^1(\bt)=H^1(\bt)$, whenever $f\in C^\infty(\bt)$ with 
$0<c\leq f(z)\leq1/c$, property satisfied by all $\d_n$, $n\in\bz$. It is also easily seen that $D^\s_n$ is symmetric on $H^1(\bt)\oplus H^1(\bt)$. Then it is enough to check 
$\cd_{(D^\s_n)^*}\subset H^1(\bt)\oplus H^1(\bt)$. Indeed, let $\eta=(\eta_1,\eta_2)\in\cd_{(D^\s_n)^*}$ with $\| \eta \| = 1$. Then for vectors $\xi\in\cd_{D^\s_n}$ of the form $(0,\tilde\xi)$ and 
$(M_{\d_n}\tilde\xi,0)$ with $\tilde\xi\in H^1(\bt)$, we get
$$
|\langle D^\s_n\xi,\eta\rangle|=\left\{\begin{array}{ll}
                     |\langle L_n\tilde\xi,M_{\d^{-1}_n}\eta_1\rangle| & \leq C\|\tilde\xi\| \\
                      |\langle L^*_n\tilde\xi,\eta_2\rangle| & \leq C\|\d_n\|_\infty\|\tilde\xi\| 
                    \end{array}.
                    \right.
$$
Thus, $M_{\d^{-1}_n}\eta_1$ hence $\eta_1$, and $\eta_2$ are in $H^1(\bt)$, and therefore all $D^\s_n$ are selfadjoint on $H^1(\bt)\oplus H^1(\bt)$.

Concerning $D^\s$, which is the direct sum of the $D^\s_n$, we note that it is symmetric on $\cd_{D^\s}$. As before, it is enough to show that $\cd_{(D^\s)^*}\subset\cd_{D^\s}$. Suppose 
$\eta\in\cd_{(D^\s)^*}$. By considering elements of the form $\xi=\tilde\xi\d_{n}$, $\tilde\xi\in H^1(\bt)\oplus H^1(\bt)$, we easily show that $\eta_n\in H^1(\bt)\oplus H^1(\bt)$ for each $n\in\bz$. 

By taking into account the previous facts, we compute for $\xi=(\xi_n)_{n\in\bz}\in\cd_{D^\s}$ with finite support in $n$, and $\eta\in\cd_{(D^\s)^*}$,
$$
\bigg|\sum_{n\in\bz}\langle\xi_n,D^\s_n\eta_n\rangle\bigg|=\bigg|\sum_{n\in\bz}\langle D^\s_n\xi_n,\eta_n\rangle\bigg|=\big|\langle D^\s\xi,\eta\rangle\big|\leq C\|\xi\|.
$$
Therefore, $\sum_{n\in\bz}\big\|D^\s_n\eta_n\big\|^2<+\infty$, that is $\eta\in\cd_{D^\s}$.
\end{proof}
We want to check whether $D^\s$ has compact resolvent. To this aim, we see that 
\begin{equation}
\label{nsssd}
D^\s=e^KDe^K=\bigoplus_{n\in\bz}e^{K_n}D_ne^{K_n},
\end{equation}
where
$$
K=\begin{pmatrix} 
\bigoplus_{n\in\bz}M_{-\ln\d_n}& 0 \\
	0&0 \\
     \end{pmatrix}
     =\bigoplus_{n\in\bz}K_n
     =\bigoplus_{n\in\bz}\begin{pmatrix} 
M_{-\ln\d_n}& 0 \\
	0&0 \\
     \end{pmatrix}.
$$
Notice that, for the untwisted Dirac operator we have
\begin{equation*}
D_n=\sum_{m\in\bz}\sqrt{m^2+n^2}(P_{m,n}^+-P_{m,n}^-),
\end{equation*}
where $P_{m,n}^+,P_{m,n}^-$ are finite range projections with uniformly bounded range dimension. In addition, $D_n$ is invertible for each $n\neq0$. After defining (with an abuse of notation)
\begin{equation*}
D^{-1}_0:=D^{-1}_0P^\perp_{{\rm Ker}(D_0)},
\end{equation*}
we see that each $D_n$ is invertible with bounded inverse
\begin{equation}
\label{0df010}
D_0^{-1}=\sum_{m\in\bz\backslash\{0\}}\frac{P_{m,0}^+-P_{m,0}^-}{|m|},\,\,
D_n^{-1}=\sum_{m\in\bz}\frac{P_{m,n}^+-P_{m,n}^-}{\sqrt{m^2+n^2}},\,\,
n\in\bz\backslash\{0\}.
\end{equation}
This immediately yields the well known fact that $D^{-1}$, and therefore the $D^{-1}_n$, $n\in\bz$, are compact operators. 

Concerning $D^\s$, after defining as before
\begin{equation}
\label{1df1}
(D^{\s}_0)^{-1}:=(D^{\s}_0)^{-1}P^\perp_{{\rm Ker}(D^{\s}_0)},
\end{equation}
the $D^\s_n$ are invertible with bounded inverse given by
\begin{equation}
\label{1czadf18}
     (D_n^\s)^{-1}=\begin{pmatrix} 
	 0 &M_{\d_n}(L_n^*)^{-1}\\
	L_n^{-1}M_{\d_n}& 0\\
     \end{pmatrix}.
\end{equation}
Therefore, $D^\s$ has always an inverse given by
\begin{equation}
\label{1czadf19}
(D^\s)^{-1}=\bigoplus_{n\in\bz}(D_n^\s)^{-1}
\end{equation}
which is bounded ({\it i.e.} $\car_{D^\s}=\ch_\om\oplus\ch_\om$ by the closed graph theorem) if and only if $\bigoplus_{n\in\bz}(D^\s_n)^{-1}$ defines a bounded operator, which happens if and only if
\begin{equation}
\label{1czadf1}
\sup_{n\in\bz}\big\|(D^\s_n)^{-1}\|<+\infty.
\end{equation}
As $D^\s$ is selfadjoint, it has compact resolvent if and only if $(D^\s)^{-1}$ is compact with the convention in \eqref{1df1} for the inverse, provided \eqref{1czadf1} is satisfied. For such a purpose, we are interested in the asymptotic of $\|(D_n^\s)^{-1}\|$ for $n\to\infty$, in the natural direct sum decomposition \eqref{1czadf19} of $D^\s$.
\begin{lem}
\label{gamrh}
We have $\big\|(D_n^\s)^{-1}\big\|\leq\frac{{\G}_{|n|}(T^2)}{|n|},\quad n\in\bz\backslash\{0\}$.
\end{lem}
\begin{proof}
By taking into account the definition of the growth sequence, \eqref{1czadf18}, \eqref{rndica} and \eqref{0df010}, we get
\begin{align*}
\big\|(D_n^\s)^{-1}\big\|=\big\|L_n^{-1}M_{\d_n}\big\|\leq&\big\|M_{\d_n}\big\|\big\|L_n^{-1}\big\|\leq\G_{|n|}(T^2)\big\|L_n^{-1}\big\|\\
=&\G_{|n|}(T^2)\big\|D_n^{-1}\big\|\leq\frac{\G_{|n|}(T^2)}{|n|}.
\end{align*}
\end{proof}
The main result of the present section is the following
\begin{thm}
\label{t2}
The Dirac operator $D^\s$ in \eqref{ndsd} has compact resolvent if and only if 
\begin{equation}
\label{kop}
\lim_{n\to\infty}\big\|(D^\s_n)^{-1}\big\|_{L^2(\bt, m)}=0,
\end{equation}
and if 
\begin{equation}
\label{kop01}
\G_{n}(T^2)=o(n).
\end{equation}
\end{thm}
\begin{proof}
The proof relies on the form \eqref{nsssd} of $D^\s$. First, it has compact resolvent if and only if $(D^\s)^{-1}$ is compact. As each of its direct summands $(D_n^\s)^{-1}=e^{-K_n}D_n^{-1}e^{-K_n}$ is compact, it will happen if and only if \eqref{kop} is satisfied. Indeed, if the this holds true, $(D^\s)^{-1}$ is compact being the norm limit of the sequence of compact operators $(D^\s)^{-1}P_n$, with $P_n$ the orthogonal projection onto the $n$-direct summand $L^2(\bt,m)\oplus L^2(\bt,m)$ in $\ch_\om\oplus\ch_\om$. Conversely, suppose \eqref{kop} does not hold true. Then there would exist a sequence $(\xi_k)_{k\in\bn}\subset\ch_\om\oplus\ch_\om$ of unit vectors such that for $k,l\in\bn$ $k\neq l$,
$$
(D^\s)^{-1}\xi_k\perp (D^\s)^{-1}\xi_l,\quad \big\|(D^\s)^{-1}\xi_k\big\|\geq\eps>0.
$$
Therefore, the sequence $\big((D^\s)^{-1}\xi_k\big)_{k\in\bn}$ is not relatively compact.

For the second half, if \eqref{kop01} is satisfied, Lemma \ref{gamrh} leads to \eqref{kop}, and thus $(D^\s)^{-1}$ is compact.
\end{proof}

\section{The modular spectral triple}
\label{mstc}

\noindent
The present section is aimed to build a reasonable deformed spectral triple whose twist ({\it i.e.} the twist of the deformed commutator with the associated Dirac operator) is constructed by using the Tomita modular operator. 

To make the reasoning meaningful at the first stage, we fix any element $A\in\pi_\om(\ba_{2\a})''$ in the set of the analytic elements w.r.t. the modular action $\s^\om$ associated with the state $\om$ ({\it e.g.} Section 10.16 of \cite{SZ}), and define on the doubled representation space $\ch_\om\oplus\ch_\om$ in \eqref{accam},
$$
\S_t(A):=
\begin{pmatrix} 
	 \D^{\imath t}A\D^{-\imath t} &0\\
	0& A\\
     \end{pmatrix},\quad t\in\br.
$$
With a slight abuse of notation, we also define the deformed commutator as follows:
$$
\cd_L^{\s}(A)\equiv\imath\big[D^\s, A\big]_{\s}:=\imath\big(D^\s\S_{-\imath}(A)-\S_{\imath}(A)D^\s\big),
$$
where $\S_{z}$ denotes the analytic continuation to $z\in\bc$ of the one-parameter group $t\in\br\mapsto\S_{t}\in\cb\big(\cb(\ch_\om\oplus\ch_\om)\big)$.

A straightforward computation then yields
\begin{equation}
\label{spetr112}
\cd^{\s}_L(A)=\imath
\begin{pmatrix} 
	0&\D^{-1}[L,A]\\
	[L^*,A]\D^{-1}& 0\\
     \end{pmatrix}.
\end{equation}
Such a definition of deformed commutator would give a real map: $\cd^{\s}_L(A)^*=\cd^{\s}_L(A^*)$, provided that $\cd^{\s}_L(A)$ uniquely defines a bounded operator acting on 
$\ch_\om\oplus\ch_\om$. 
Notice also that \eqref{spetr112} can be meaningful even if $A$ is not an analytic element.

We adopt directly \eqref{spetr112} as the definition of the deformed derivation, and look at elements $A\in\pi_\om(\ba_{2\a})$ such that $\cd^{\s}_L(A)$ uniquely defines bounded operators acting on $\ch_\om\oplus\ch_\om$ according to Definition \ref{fmst}.

Unfortunately, when dealing with the natural candidate for the deformed Dirac operator appearing in Section \ref{dirac}, we meet an unavoidable obstruction arising from the commutation relation
\begin{equation}
\label{medr0}
\big[L,\l^l\big]=-l\l^l,\quad l\in\bz,
\end{equation}
and therefore
\begin{equation}
\label{medr}
\D^{-1}\big[L,\l^l\big]=-l\D^{-1}\l^l,\quad l\in\bz,
\end{equation}
where $\l\in\pi_\om(\ba_{2\a})$ is the one-step shift given in \eqref{lash}.
Due to the presence of the unbounded operator $\D^{-1}$, the latter can never be bounded in the type $\ty{III}$ case. This is connected with the asymmetric role of the abelian algebras generated 
by the Weyl operators $W(\d_{m,k}\d_{n,0})$ and $W(\d_{m,0}\d_{n,l})$,
in the construction of the states $\om$ considered in Section \ref{secstat}. 
However, we can yet exhibit the corresponding modular spectral triple associated with the Dirac operator \eqref{ndsd}, which might be of interest for some possible application.

Consider the subset $\cc_0\subset C(\bt)$ defined by 
$$
\cc_0:=\big\{f\circ\mathpzc{h}_{T^2}\mid f\in C^1(\bt)\big\}.
$$
For each finite subset $J\subset\bz$, consider the functions $f_k(m,n):=\widehat{F_k}(m)\d_{n,0}$ with $F_k\in \cc_0$, and $g_k(m,n):=\d_{m,0}\d_{n,k}$. We define $\cb_0\subset\cb(\bz^2)$ as the set of all functions
\begin{equation}
\label{prrcsalcz}
f:=\sum_{k\in J}(f_k*_{2\a}g_k),\quad J\subset\bz\,\,\,\text{finite}\,.
\end{equation}
We also define
$$
\ba_{2\a}^{oo}:=\big\{W(f)\mid f\in\cb_0\big\}.
$$
By taking into account
$$
f\circ\mathpzc{h}_{T^2}\circ R^{2n}=f\circ T^{2n}\circ\mathpzc{h}_{T^2},
$$
$f\circ T^{2n}\in C^1(\bt)$ provided that $f\in C^1(\bt)$. Therefore,
$\ba_{2\a}^{oo}$ is a unital $*$-algebra, which is dense in $\ba_{2\a}$ by construction.
\begin{thm}
\label{sczptr}
For each orientation preserving $C^\infty$-diffeomorphism $\mathpzc{f}$ of the unit circle $\bt$ with $\r(\mathpzc{f})=2\a$, the following assertions hold true.
\begin{itemize}
\item[(i)] If $A=\pi_\om(W(f))$ with $f(m,n):=\hat G(m)\d_{n,0}$, $G\in\cc_0$, then $\cd^{\s}_L(A)$ defines a bounded operator acting on $\ch_\om\oplus\ch_\om$.
\item[(ii)] If $A\in\pi_\om(\ba_{2\a}^{oo})$ and $N\in\bn$, $(P_N\oplus P_N)\cd^{\s}_L(A)$ defines a bounded operator acting on $\ch_\om\oplus\ch_\om$, where $P_N$ is the orthogonal projection onto $\bigoplus_{|k|\leq N}L^2(\bt, m)$.
\end{itemize}
\end{thm}
\begin{proof}
By Denjoy theorem, the diffeomorphism $\mathpzc{f}$ is conjugate to the rotation $R_{2\a}$ as established in {\bf (A)} given in Section \ref{5dfm}. Therefore, $\mathpzc{f}=T^2$ with
$T:=\mathpzc{h}_\mathpzc{f}\circ R_{\a}\circ \mathpzc{h}_\mathpzc{f}^{-1}$.

(i) Consider $F\in C^1(\bt)$ and set $f(m,n):=\widehat{F\circ\mathpzc{h}_{T^2}}(m)\d_{n,0}$. We note that $\pi_\om(W(f))_n=M_{F\circ T^{2n}}$. Then, setting $A=\pi_\om(W(f))$, $[L,A]$ is diagonal, providing the multiplication by the functions
\begin{equation}
\label{pfircl}
[L,A]_n(z)=\imath z\frac{\di\,}{\di z}F(T^{2n}(z))=\imath zDT^{2n}(z)(DF)(T^{2n}(z)).
\end{equation}
This implies that the deformed commutator is given by the multiplication by the functions
$$
(\D^{-1}[L,A])_n(z)=\imath T^{2n}(z)(DF)(T^{2n}(z)),
$$
which are uniformly bounded in $n\in\bz$ as $F\in C^1(\bt)$. Thus, the deformed commutator \eqref{spetr112} is meaningful for $A=\pi_\om(W(f))$ as a bounded operator acting on $\ch_\om\oplus\ch_\om$.

(ii) If $f$ is given by \eqref{prrcsalcz}, then 
$$
\pi_\om\big(W(f)\big)=\sum_{k\in J}\pi_\om\big(W(f_k)\big)\l^k.
$$
The claim then follows from (i), taking into account \eqref{medr}, that $P_N$ commutes with $\D$, and finally that
$$
\big\|P_N\D^{-1}\big\|=\max_{n\leq N}\G_n(T^2)<+\infty.
$$
\end{proof}
Therefore, concerning the naturally associated spectral triple, we have
\begin{cor}
\label{enni}
With the notations of Theorem \ref{sczptr}, consider $\om\equiv\om_\m\in\cs(\ba_{2\a})$ given in \eqref{ommu} for $\m=m\circ \mathpzc{h}_\mathpzc{f}$, and $L$ in \eqref{eldcz}.
\begin{itemize}
\item[(i)] The triplet $\big(\om,\pi_\om(\ba_{2\a}^{oo}), L\big)$ satisfies (i), (ii) and (iv) of Definition \ref{fmst}, provided that $\G_n(f)=o(n)$.
\item[(ii)] The type of the hyperfinite von Neumann factor $\pi_\om(\ba_{2\a})''$ is determined by the Krieger-Araki-Woods ratio set $r([m],\mathpzc{f})$, if it is not of type $\ty{II_1}$.
\end{itemize}
\end{cor}
\begin{proof}
Concerning (iv) of Definition \ref{fmst}, it easily follows from the definition of $\ba_{2\a}^{oo}$, \eqref{pfircl} and \eqref{medr0}.
The remaining part follows from Theorem \ref{sczptr} and Theorem \ref{thmcz}.
\end{proof}
We end the present section by noticing that, in our situation, $\a$ diophantine produces always the $\ty{II_1}$ case, whereas the non type $\ty{II_1}$ is covered by Proposition \ref{catsum} for Liouville numbers.

\section{Another proposal for the modular spectral triple}
\label{pmstc}

\noindent
Due to the obstruction \eqref{medr}, it seems that the deformed commutator \eqref{spetr112} might not provide sufficiently many bounded operators for any reasonable dense set of $\ba_{2\a}$. In the present section, we explain how to overcome this unpleasant feature by modifying the definition of the Dirac operator.

To this aim, fix any orientation preserving diffeomorphism $\mathpzc{f}$ of the circle satisfying {\bf(A)} w.r.t. a rotation by the angle $4\pi\a$, and set $T:=\mathpzc{h}_\mathpzc{f}\circ R_{\a}\circ \mathpzc{h}_\mathpzc{f}^{-1}$ ({\it cf.} Section \ref{t3rp}). With
\begin{equation}
\label{zdfz}
\mathrm{L}=\bigoplus_{n\in\bz}\bigg(\imath z\frac{\di\,\,}{\di z}-a_nI\bigg),
\end{equation}
and
$$
a_n:=\sign(n)\sum_{l=1}^{|n|}\frac1{\G_{l-\frac{1-\sign(n)}2}(T^2)},\quad n\in\bz,
$$
we define 
$\mathrm{D}=\bigoplus_{n\in\bz}\mathrm{D}_n$, where
$$
\mathrm{D}_n=\begin{pmatrix} 
	 0 &\mathrm{L}_n\\
	\mathrm{L}_n^*& 0\\
     \end{pmatrix}:=\begin{pmatrix} 
	 0 &\left(\imath z\frac{\di\,\,}{\di z}-a_nI\right)\\
	\left(-\imath z\frac{\di\,\,}{\di z}-a_nI\right)& 0\\
     \end{pmatrix}.
$$
The associated deformed Dirac operator assumes the form
$$
\mathrm{D}^\s=\bigoplus_{n\in\bz}\mathrm{D}_n^\s=\bigoplus_{n\in\bz}\begin{pmatrix} 
	 0 &M_{\d_n^{-1}}\mathrm{L}_n\\
	\mathrm{L}_n^*M_{\d_n^{-1}}& 0\\
     \end{pmatrix}
     =\begin{pmatrix} 
	 0 &\D^{-1}\mathrm{L}\\
	\mathrm{L}^*\D^{-1}& 0\\
     \end{pmatrix}.
$$
Notice that such Dirac operators (twisted and untwisted) coincide with the canonical ones for the representation associated with the trace.  

As in Section \ref{dirac} for the untwisted Dirac operator $D$, we recognise that $\mathrm{D}_n$ is invertible when $n\neq0$. We set (with a little abuse of notation)
\begin{equation*}
\mathrm{D}^{-1}_0:=\mathrm{D}^{-1}_0P^\perp_{{\rm Ker}(\mathrm{D}_0)},
\end{equation*}
and see that each $\mathrm{D}_n$ is invertible with bounded inverse
\begin{equation*}
\mathrm{D}_0^{-1}=\sum_{m\in\bz\backslash\{0\}}\frac{P_{m,0}^+-P_{m,0}^-}{|m|},\quad
\mathrm{D}_n^{-1}=\sum_{m\in\bz}\frac{P_{m,n}^+-P_{m,n}^-}{\sqrt{m^2+a_n^2}},
\end{equation*}
where $P_{m,n}^\pm$ are the selfadjoint projections onto the corresponding sspectral subspaces. This immediately implies that $\mathrm{D}^{-1}_n$ is a compact operator for each $n\in\bz$.  

The same proof as that in Proposition \ref{dsukz} shows that $\mathrm{D}^\s$ is selfadjoint on the domain 
$$
\cd_{\mathrm{D}^\s}:=\bigg\{\xi\in\bigoplus_{n\in\bz}\cd_{\mathrm{D}^\s_n}\mid\sum_{n\in\bz}\|\mathrm{D}^\s_n\xi_n\|^2<+\infty\bigg\},
$$
where $\cd_{\mathrm{D}^\s_n}=H^1(\bt)\oplus H^1(\bt)$ for each $n\in\bz$. In addition, with
$$
(\mathrm{D}^{\s}_0)^{-1}:=(\mathrm{D}^{\s}_0)^{-1}P^\perp_{{\rm Ker}(\mathrm{D}^{\s}_0)},
$$
all $\mathrm{D}^{\s}_n$ have a bounded inverse, and therefore $\cd_{\mathrm{D}^\s}$ has a bounded inverse if and only if
$$
\sup_{n\in\bz}\big\|(\mathrm{D}^\s_n)^{-1}\|<+\infty.
$$
Finally, concerning the compactness of the resolvent of $\mathrm{D}^\s$, we look at the asymptotic of $\|(\mathrm{D}_n^\s)^{-1}\|$ for $n\to\infty$ obtaining
\begin{thm}
\label{t22}
The Dirac operator $\mathrm{D}^\s$ has compact resolvent if and only if 
\begin{equation*}
\lim_{n\to\infty}\big\|(\mathrm{D}^\s_n)^{-1}\big\|_{L^2(\bt, m)}=0,
\end{equation*}
and if 
$$
\G_{n}(T^2)=o(\ln n).
$$
\end{thm}
\begin{proof}
The proof of the first assertion proceeds exactly in the same way as the analogous one in Theorem \ref{t2}. For the second half, reasoning as in Lemma \ref{gamrh}, for $|n|$ sufficiently large we have
\begin{align*}
\big\|(\mathrm{D}_n^\s)^{-1}\big\|\leq&\frac{\G_{|n|}(T^2)}{\sum_{l=1}^{|n|}\frac1{\G_{l-\frac{1-\sign(n)}2}(T^2)}}
=\frac1{\frac1{\G_{|n|}(T^2)}\sum_{l=1}^{|n|}\frac1{\G_{l-\frac{1-\sign(n)}2}(T^2)}}\\
\leq&\frac1{\frac1{\ln |n|}\sum_{l=2}^{|n|}\frac1{\ln l}}
\sim\frac1{\frac1{\ln |n|}\int_{{}_2}^{{}^|n|}\frac{\di x}{\ln x}}.
\end{align*}
The assertion then follows since, for the 
logarithmic integral function,
$$
{\rm Li}(x)\geq\frac{x}{\ln x}-\frac{2}{\ln 2}.
$$
\end{proof}
As in \eqref{spetr112}, we set for the associated deformed commutator,
\begin{equation}
\label{spetr232}
\cd^{\s}_\mathrm{L}(A)=\imath
\begin{pmatrix} 
	&\D^{-1}[\mathrm{L},A]\\
	[\mathrm{L}^*,A]\D^{-1}& 0\\
     \end{pmatrix}.
\end{equation}
We now pass to show that $\cd^\s_\mathrm{L}(\pi_\om(a))$ uniquely defines a bounded operator, provided $a\in\ba^{oo}_{2\a}$. For such $A\in\pi_\om(\ba_{2\a}^{oo})$, we first notice that 
$$
A\cd_{\mathrm{L}}\subset\cd_{\mathrm{L}},\quad A\cd_{\mathrm{L^*}}\subset\cd_{\mathrm{L^*}}
$$
({\it i.e.} (iv) in Definition \ref{fmst} is satisfied). Therefore,
\begin{align*}
\cd_{\cd^\s_\mathrm{L}(A)}=&\bigg\{\xi,\eta\in\bigoplus_{n\in\bz}H^1(\bt)\mid\sum_{n\in\bz}\big(\big\|([\mathrm{L}^*,A]\D^{-1}\xi)_n\big\|^2\\
+&\big\|M_{\d_n^{-1}}([\mathrm{L},A]\eta)_n\big\|^2\big)<+\infty\bigg\}.
\end{align*}
We then get the following
\begin{thm}
\label{mdefca}
If $A\in\pi_\om(\ba_{2\a}^{oo})$, then the deformed commutator \eqref{spetr232} uniquely defines a bounded operator acting on $\ch_\om\oplus\ch_\om$.
\end{thm}
\begin{proof}
By taking into account the commutation rule
$$
[\mathrm{L}^\#, AB]=A[\mathrm{L}^\#, B]+[\mathrm{L}^\#, A]B,
$$
where $\mathrm{L}^\#$ stands for $\mathrm{L}$ or $\mathrm{L}^*$, we can reduce the matter to the algebraic generators of $\ba_{2\a}^{oo}$. 

First, for $A=\pi_\om(W(f))$ with $f$ as in (i) of Theorem \ref{sczptr}, we have $\cd^{\s}_\mathrm{L}(A)=\cd^{\s}_{L}(A)$ for $L$ given in \eqref{eldcz}, and thus the claim follows. 

Second, for $A=\l$ the one-step shift given in \eqref{lash}, together with its adjoint $\l^{-1}$, we first note that for each $n\in\bn$,
$$
\G_{n+1}(T^2)/\G_{n}(T^2),\,\G_{n}(T^2)/\G_{n+1}(T^2)\leq\G_{1}(T^2),
$$
which can be checked by the simple formulas
$$
D\mathpzc{f}^{n+1}=D\mathpzc{f}\times\big((D\mathpzc{f}^n)\circ\mathpzc{f}\big),\quad D\mathpzc{f}^{n-1}=D\mathpzc{f}^{-1}\times\big((D\mathpzc{f}^n)\circ\mathpzc{f}\big).
$$
For each $n\in\bz$, we notice that
$$
|a_{n-1}-a_n|=\frac1{\G_{|n|}(T^2)},
$$
and then compute
\begin{align*}
\left\|\left(\D^{-1}[\mathrm{L},\l]g\right)_n\right\|&=\left\|(a_{n-1}-a_n)M_{\d_n^{-1}}\big(\l g\big)_n\right\|\\
\leq&|a_{n-1}-a_n|\G_{|n|}(T^2)\|g\|=\|g\|,
\end{align*}
with the analogous one 
$$
\left\|\left([\mathrm{L}^*,\l]\D^{-1}g\right)_n\right\|\leq|a_{n-1}-a_n|\G_{|n-1|}(T^2)\|g\|\leq \G_{1}(T^2)\|g\|.
$$
The formulas for $\l^{-1}$ follow from the previous ones by taking the $*$-operation.

Hence, $\cd^{\s}_\mathrm{L}(\l)$ defines a bounded operator as well.
\end{proof}
\begin{cor}
\label{vdfv}
With the same notations in Theorem \ref{sczptr}, consider $\om\equiv\om_\m\in\cs(\ba_{2\a})$ given in \eqref{ommu} for $\m=m\circ \mathpzc{h}_\mathpzc{f}$, and $\mathrm{L}$ in \eqref{zdfz}.
\begin{itemize}
\item[(i)] The triplet $\big(\om,\pi_\om(\ba_{2\a}^{oo}), \mathrm{L}\big)$ satisfies (i)-(iv) of Definition \ref{fmst}, provided that $\G_n(f)=o(\ln n)$.
\item[(ii)] The type of the hyperfinite von Neumann factor $\pi_\om(\ba_{2\a})''$ is determined by the Krieger-Araki-Woods ratio set $r([m],\mathpzc{f})$, if it is not of type $\ty{II_1}$.
\end{itemize}
\end{cor}
\begin{proof}
It immediately follows  collecting together Theorem \ref{t22} and Theorem \ref{mdefca}.
\end{proof}
Therefore, by Proposition \ref{catsum} we can thus exhibit non type $\ty{II_1}$ modular spectral triples of the noncommutative torus $\ba_{2\a}$ satisfying all requirements in Definition \ref{fmst}, provided that $\a$ satisfies the fast approximation property {\bf(UL)}.

\section{Outlook}
\label{ocfl}

\noindent
For the convenience of the reader, we end the present paper by writing down a list, which is very far from being complete, of some open questions related to the analysis carried out in the present paper.
\begin{itemize}
\item[(a)] We point out that the analysis concerning all natural properties encoded in the new spectral triples introduced in the present paper in a type $\ty{III}$ setting, deserves a detailed investigation as explained in \cite{CM, V}. As a first step, we mention the definition/construction of the related (probably even twisted) Fredholm  modules.
\end{itemize}
Concerning the properties of the orientation preserving diffeomorphisms $\mathpzc{f}$ of the circle as those considered in the present paper, the construction of non type $\ty{II_1}$ representations and the corresponding modular spectral triples relies on the combined control of two crucial properties: the ratio set $r([m],\mathpzc{f})$ of the classical dynamical system 
$(\bt,\mathpzc{f},m)$, and the growth sequence $\G_n(\mathpzc{f})$. 
As intermediate results of self containing interest, we have shown that $\G_n(\mathpzc{f})=o(n)$ for all diffeomorphisms in Proposition 2.1 of \cite{M}.
Moreover, if the involved Liouville number $\a$ satisfies the faster approximation condition ${\bf (UL)}$ introduced in Section \ref{sec2}, it is possible to construct by the same methods as in \cite{M}, diffeomorphisms $\mathpzc{f}$ of the unit circle such that the ratio set $r([m],\mathpzc{f})$ is of preassigned kind, and in addition $\G_n(\mathpzc{f})=o(\ln n)$. We argue that the analysis of \cite{M}, on which Section \ref{difa} is based, might be further extended as follows.
\begin{itemize}
\item[(b)] For each positive monotone sequence $b_n\to+\infty$ and ratio set of preassigned type ${\rm F}$, there would exist a set of Liouville numbers $\a\in(0,1)$ (indeed a dense set in $[0,1]$ depending on the sequence $(b_n)_n$) and a set of diffeomorphisms $\mathpzc{f}$ of the circle fulfilling {\bf(A)}, for which $r([m],\mathpzc{f})={\rm F}$ and $\G_n(\mathpzc{f})=o(b_n)$.
\end{itemize}
In this way, we would also provide a generalisation of Theorem 2 in \cite{W}, which asserts the control on the growth sequence without any control on the ratio set.

For orientation preserving diffeomorphisms $\mathpzc{f}$ of the circle as above, with the property that $\r(\mathpzc{f})=\a$ for an irrational number $\a$, we have looked at the unique invariant measure $\m_\mathpzc{f}$ given in \eqref{uinma}. Denote by $F_\a$ the set of such diffeomorphisms, and for $d\in[0,1]$,
$$
S^d_\a:=\big\{\mathpzc{f}\in F_\a\mid \dim_H(\m_\mathpzc{f})=d\big\},
$$
where $\dim_H$ stands for the Hausdorff dimension, see {\it e.g.} \cite{M1}. It is well known that $\dim_H(\m_\mathpzc{f})<1$ implies that $\m_\mathpzc{f}\perp m$, and thus $\mathpzc{f}$ is not of type $\ty{II_1}$. 
\begin{itemize}
\item[(c)] It would be of interest to prove or disprove at various levels the following conjecture that links the Hausdorff dimension of the unique invariant measure to the ratio set:
\begin{itemize}
\item[] $\mathpzc{f}\in S^1_\a\iff\mathpzc{f}$ is of type $\ty{II_1}$,
\item[] $\mathpzc{f}\in\bigcup_{d\in(0,1)}S^d_\a\iff\mathpzc{f}$ is of type $\ty{II_\infty}$,
\item[] $\mathpzc{f}\in S^0_\a\iff\mathpzc{f}$ is of type $\ty{III}$.
\end{itemize}
\end{itemize}
Concerning diophantine numbers $\a\in(0,1)$, the method developed in the present paper might suggest the way to construct non type $\ty{II_1}$ representations, together with the corresponding modular spectral triples.
As we have previously shown, the construction of modular spectral triples with the desired properties is based on the fact that the Borel automorphisms of \cite{K4} have to be sufficiently smooth. It relies on the knowledge of the asymptotic of the growth sequence of such diffeomorphisms.
The natural candidates would be orientation preserving (non sufficiently smooth) $C^1$-diffeomorphisms $\mathpzc{f}$ of the circle with $\r(\mathpzc{f})=\a$ which are topologically transitive, 
and therefore satisfy {\bf (A)} for $\mathpzc{h}_\mathpzc{f}$ a non smooth homeomorphism. More precisely,
\begin{itemize}
\item[(d)] for diophantine numbers $\a\in(0,1)$, to construct examples of such $C^1$-diffeomorphisms for any given type $\ty{II_\infty}$, and $\ty{III_\l}$, $\l\in[0,1]$, together with an appropriate control of the growth sequence.
\end{itemize}
Finally, we would like to mention the following potential applications to the physical scenario.
\begin{itemize}
\item[(e)] There might be possible applications of the new non type $\ty{II_1}$ representations in the context of the quantum Hall effect, see {\it e.g.} \cite{BS, MP}. 
\item[(f)] We also note the possible generalisation of the present analysis to arbitrary CCR algebras based on a locally compact abelian group equipped with a symplectic form ({\it cf.} \cite{Z}), and in particular to those describing physical systems with finite degrees of freedom.
\end{itemize}

\section{Appendix}
\label{ppa}

\noindent
The dense $*$-algebra $\ba^{oo}_{2\a}$ appearing in the definition of the modular spectral triple associated with the noncommutative torus cannot be closed under the entire functional calculus. For untwisted spectral triples, the algebra of the functions for which the derivation generated by the associated (untwisted) Dirac operator provides a bounded operator, can be enlarged to include at least the smooth ({\it i.e.} $C^\infty$) functional calculus, see {\it e.g.} \cite{BC}. Unfortunately, there is no investigation for the same question relatively to general twisted spectral triples. Therefore, we will show how to enlarge the algebra $\ba^{oo}_{2\a}$ in Corollary \ref{vdfv}
in order to satisfy such a requirement, perhaps expected, of smoothness.

Consider $f\in\cb(\bz^2)$ such that $W(f)\in\ba_{2\a}$. For $k\in\bn$, $l=0,1$, define the following sequence of seminorms
\begin{equation}
\label{semi}
\r_{k,l}\big(W(f)\big):=\sup_{n\in\bz}\left\{(|n|+1)^k\left\|D^l\left(\widecheck{f^{(n)}}\circ R^{-n}\circ \mathpzc{h}_{T^2}^{-1}\right)\right\|_\infty\right\},
\end{equation}
provided $\widecheck{f^{(n)}}\circ R^{-n}\circ \mathpzc{h}_{T^2}^{-1}\in C^1(\bt)$, $n\in\bz$. Define
$$
\ba^{o}_{2\a}:=\bigg\{W(f)\mid\r_{k,l}\big(W(f)\big)<+\infty,\,\, k\in\bn,\, l=0,1\bigg\}.
$$
The algebra $\ba^{o}_{2\a}$ will play the role of the algebra of the ``smooth'' functions $\ca$ in Definition \ref{fmst}. Notice that, in the more interesting cases when the homomorphism $\mathpzc{h}_\mathpzc{f}$ in {\bf(A)} of Section \ref{sec2} (where $\mathpzc{f}=T^2$ with $\r(\mathpzc{f})=2\a$ as usual) is not smooth, which necessarily happens in non type $\ty{II_1}$ cases, the ``smooth'' generator of the ``coordinate functions'' acting on $L^2(\bt,m)$ is $M_z=\pi_\om(W(f_1))$, with 
$$
f_1(m,n)=\widehat{\mathpzc{h}_\mathpzc{f}}(m)\d_{n,0} \,,
$$
instead of 
$$
\pi_\om(W(f_2))=M_{\mathpzc{h}_\mathpzc{f}^{-1}(z)},
$$
with $f_2(m,n)=\d_{m,1}\d_{n,0}$. If $\mathpzc{h}_\mathpzc{f}$ is smooth ({\it i.e.} for the type $\ty{II_1}$ case), both functions generate dense $*$-algebras of the smooth functions. In the non type $\ty{II_1}$ cases, both generators generate the maximal abelian subalgebra consisting of a copy of the algebra of continuous functions $C(\bt)$ in $\pi_\om(\ba_{2\a})$. But the ``coordinate function'' corresponding to $\pi_\om(W(f_2))$ generates an algebra that acts through functions which are not smooth as expected.

We now provide some preparatory results.
\begin{lem} 
\label{dlct}
$\ba^{oo}_{2\a}\subset \ba^{o}_{2\a}$ is dense in $\ba^{o}_{2\a}$ in the locally convex topology generated by the collection of seminorms $\{\r_{k,l}\mid k\in\bn,\, l=0,1\}$ given in \eqref{semi}.
\end{lem}
\begin{proof}
For $W(\f)\in\ba_{2\a}^o$, then
$$
\widecheck{\f^{(n)}}=H_n\circ\mathpzc{h}_{T^2}\circ R^n,\quad n\in\bz,
$$
where $\{H_n\mid n\in\bz\}\subset C^1(\bt)$. Let $f_k(m,n)=\widehat{H_k\circ\mathpzc{h}_{T^2}}(m)\d_{n,0}$ for each $k\in\bz$, and define as in \eqref{prrcsalcz},
$\f_N:=\sum_{|k|\leq N}f_k*_{2\a}g_k$, $N\in\bn$,
where $g_k(m,n):=\d_{m,0}\d_{n,k}$. We then get
\begin{align*}
&\r_{k,l}\big(W(\f-\f_N)\big)=\sup_{|n|>N}\left\{(|n|+1)^k\left\|D^l\left(\widecheck{f^{(n)}}\circ R^{-n}\circ \mathpzc{h}_{T^2}^{-1}\right)\right\|_\infty\right\}\\
\leq&\frac{\sup_{|n|>N}\left\{(|n|+1)^{k+1}\left\|D^l\left(\widecheck{f^{(n)}}\circ R^{-n}\circ \mathpzc{h}_{T^2}^{-1}\right)\right\|_\infty\right\}}{|N|+1}
\leq\frac{\r_{k+1,l}\big(W(\f)\big)}{|N|+1}.
\end{align*}
Therefore, $\r_{k,l}\big(W(\f-\f_N)\big)\to0$ whenever $N\to+\infty$, for each fixed $k\in\bn$, $l=0,1$.
\end{proof}
\begin{lem}
\label{stsig}
For each $k\in\bn$, there exists a constant $B(k)>0$ such that
\begin{align*}
&\r_{k,0}\big(W(f*_{2\a} g)\big)\leq B(k)\r_{k\vee2,0}\big(W(f)\big)\r_{k\vee2,0}\big(W(g)\big),\\
&\r_{k,1}\big(W(f*_{2\a} g)\big)\leq B(k)\big(\r_{k\vee2,1}\big(W(f)\big)\r_{k\vee2,0}\big(W(g)\big)
+\r_{k\vee2+2,0}(f)\r_{k\vee2,1}(g)\big).
\end{align*}
\end{lem}
\begin{proof}
Fix two positive sequences 
$$
\{f_n\mid n\in\bz\}, \{g_n\mid n\in\bz\}\subset [0,+\infty).
$$
For each $m>1$, we get
\begin{align*}
&(|n|+1)^k(f*g)(n)=\left(\sup_{n\in\bz}(|n|+1)^mf_n\right)\left(\sup_{n\in\bz}(|n|+1)^mg_n\right)\\
\times&(|n|+1)^k\sum_{r\in\bz}(|r|+1)^{-m}(|n-r|+1)^{-m}.
\end{align*}
As the above identity is symmetric under the exchange $n\to-n$, we can reduce the proof to positive numbers. Moreover, as we must check the behaviour for large $n$, first we split the sum in various pieces by solving the modulus, and then estimate each piece with integrals, thus obtaining
\begin{align*}
&\sum_{r\in\bz}(|r|+1)^{-m}(|n-r|+1)^{-m}
\approx\frac2{(n+1)^m}
+\int_{-\infty}^{-1}\frac{\di x}{[x(x-n)]^m}\\
+&\int_{1}^{n-1}\frac{\di x}{[x(n-x)]^m}
+\int_{n+1}^{+\infty}\frac{\di x}{[x(x-n)]^m}.
\end{align*}
The first integral leads to
$$
\int_{-\infty}^{-1}\frac{\di x}{[x(x-n)]^m}\leq\frac1{(m-1)(n+1)^m}.
$$
Concerning the third one, we easily get the same estimate as before after an elementary change of variables:
$$
\int_{n+1}^{+\infty}\frac{\di x}{[x(x-n)]^m}\leq\frac1{(m-1)(n+1)^m}.
$$
It remains to estimate the second integral which, after a change of variables, becomes 
$$
\int_{1}^{n-1}\frac{\di x}{[x(n-x)]^m}=\frac1{n^{2m-1}}\int_{1/n}^{1-1/n}\frac{\di x}{[x(1-x)]^m}.
$$
We get
$$
n^k\int_{1}^{n-1}\frac{\di x}{[x(n-x)]^{k\vee2}}\leq K(k).
$$
Collecting the terms together, we obtain the first inequality for some $B_1(k)$. 

The second inequality follows after differentiating \eqref{vide}. In fact, first we may assume without loss of generality that the sum is finite (by restricting to a dense set). Second, by taking into account that $R^{2r}\circ \mathpzc{h}_{T^2}^{-1}=\mathpzc{h}_{T^2}^{-1}\circ T^{2r}$, $r\in\bz$, and $\|DT^{2r}\|\leq Cr^2$ ({\it cf.} Theorem 1 in \cite{W}), we compute
\begin{align*}
\bigg\|D\bigg[\widecheck{(f*_{2\a} g)^{(n)}}&\circ R^{-n}\circ\mathpzc{h}_{T^2}^{-1}\bigg]\bigg\|
\leq\sum_{l\in\bz}\left\|D\left(\widecheck{f^{(l)}}\circ R^{-l}\circ\mathpzc{h}_{T^2}^{-1}\right)\right\|\left\|\widecheck{g^{(n-l)}}\right\|\\
+&\sum_{l\in\bz}\left\|\widecheck{f^{(l)}}\right\|\left\|D\left(\widecheck{g^{(n-l)}}\circ R^{l-n}\circ\mathpzc{h}_{T^2}^{-1}\circ T^{-2l}\right)\right\|\\
\leq&\sum_{l\in\bz}\left\|D\left(\widecheck{f^{(l)}}\circ R^{-l}\circ\mathpzc{h}_{T^2}^{-1}\right)\right\|\left\|\widecheck{g^{(n-l)}}\right\|\\
+&C\sum_{l\in\bz}l^2\left\|\widecheck{f^{(l)}}\right\|\left\|D\left(\widecheck{g^{(n-l)}}\circ R^{l-n}\circ\mathpzc{h}_{T^2}^{-1}\right)\right\|.
\end{align*}
The claim then follows by setting $B(k):=B_1(k)(1\vee C)$.
\end{proof}
\begin{prop}
\label{bcezc}
$\ba^{o}_{2\a}\subset\ba_{2\a}$ is a dense $*$-algebra closed under the entire functional calculus.
\end{prop}
\begin{proof}
By Lemma \ref{stsig}, it follows immediately that $\ba^{o}_{2\a}$ is an algebra which is automatically closed also under the star operation since 
$\widecheck{(f^\star)^{(n)}}=\overline{\widecheck{f^{(-n)}}}$. Since $\ba^{oo}_{2\a}\subset\ba^{o}_{2\a}$ and the former algebra is dense in $\ba_{2\a}$ by Lemma \ref{dlct}, the latter algebra will be  
dense as well.

Let now
$$
F(w)=\sum_{r=0}^{+\infty}F_rw^r
$$
be an entire function. Denoting by $f^{*_{2\a},n}$ the $n$-times twisted convolution, we write
$$
F\big(W(f)\big)=\sum_{r=0}^{+\infty}F_rW(f^{*_{2\a},r})=W(G),
$$
with
$$
G:=\sum_{r=0}^{+\infty}F_rf^{*_{2\a},r}.
$$
We define $|F|(w):=\sum_{r=0}^{+\infty}|F_r|w^{r}$, and observe that $|F|$ is entire as well, and that $|F|\lceil_{[0,+\infty)}$ is positive.
We then obtain, again by Lemma \ref{stsig},
\begin{align*}
\r_{k,0}\big(W(G)\big)\leq&\frac1{B(k)}\sum_{r=0}^{+\infty}|F_r|\big(B(k)\r_{k\vee2,0}\big(W(f)\big)\big)^r=\frac{|F|\big(B(k)\r_{k\vee2,0}\big(W(f)\big)\big)}{B(k)},\\
\r_{k,1}\big(W(G)\big)\leq&\r_{k\vee2,1}(f)\sum_{r=1}^{+\infty}r|F_r|\big[B(k)\big(\r_{k\vee2,0}\big(W(f)\big)\vee\r_{k\vee2+2,0}\big(W(f)\big)\big)\big]^{r-1}\\
=&\r_{k\vee2,1}(f)|F|'\big(B(k)\big(\r_{k\vee2,0}\big(W(f)\big)\vee\r_{k\vee2+2,0}\big(W(f)\big)\big)\big).
\end{align*}
\end{proof}
The following result formalises the fact that if $a\in\ba^{o}_{2\a}$, then $\cd^{\s}_\mathrm{L}(\pi_\om(a))$ uniquely defines a bounded operator.
\begin{prop}
\label{bcezc1}
If $W(f)\in\ba^{o}_{2\a}$, then
\begin{align*}
\big\|\cd^{\s}_\mathrm{L}\big(\pi_\om(W(f)\big)\big\|\leq&\bigg(\sum_{n\in\bz}\frac1{(1+|n|)^2}\bigg)\r_{2,1}\big(W(f)\big)\\
+&\G_1(T^2)\bigg(\sum_{n\in\bz}\frac{|n|}{(1+|n|)^3}\bigg)\r_{3,0}\big(W(f)\big).
\end{align*}
\end{prop}
\begin{proof}
By Lemma \ref{dlct}, $\ba^{oo}_{2\a}$ is dense in $\ba^{o}_{2\a}$ in the locally convex topology generated by the seminorms \eqref{semi}. It is then enough to prove the assertion for elements in the former set which, for any arbitrarily large but finite subset $J\in\bz$, have the form 
$$
A=\sum_{k\in J}\pi_\om(W(f_k))\l^k,
$$
where $f_k(m,n)=\widehat{F_k}(m)\d_{n,0}$, $F_k\in\cc_0$ ({\it i.e.}, $F_k=H_k\circ\mathpzc{h}_{T^2}$ with $H_k\in C^1(\bt)$). Indeed, in this case $A=\pi_\om(W(f))$, with $f(m,n)$ given by \eqref{prrcsalcz}, and
$$
\widecheck{f^{(n)}}=H_n\circ\mathpzc{h}_{T^2}\circ R^n,\quad n\in\bz,
$$
as explained in the proof of Proposition \ref{dlct}. We also notice that $\mathrm{L}=\mathrm{L}_1+\mathrm{L}_2$ with
$$
[\mathrm{L}_1,\l]=[\mathrm{L}_2,\pi_\om(W(f_k))]=0,
$$
$$
[\mathrm{L}_2,\l^k]=\sum_{s=0}^{k-1}\l^s[\mathrm{L}_2,\l]\l^{k-s-1},\quad k>0,
$$ 
and the analogous one for $k<0$. By taking into account the definition of $\D$ and $\mathrm{L}$, we straightforwardly conclude that
$$
\big\|\D^{-1}[\mathrm{L},A]\big\|\leq\sum_{n\in J}\|DH_n\|_{\infty}
+\sum_{n\in J}|n|\|H_n\|_{\infty}.
$$
Reasoning as the proof of Theorem \ref{mdefca}, we obtain
$$
\big\|[\mathrm{L}^*,A]\D^{-1}\big\|\leq\sum_{n\in J}\|DH_n\|_{\infty}
+\G_1(T^2)\sum_{n\in J}|n|\|H_n\|_{\infty},
$$
and the assertion follows.
\end{proof}
We end by noticing that Proposition \ref{bcezc} and Proposition \ref{bcezc1} can be viewed as the counterpart of (ii) in Example 6.5 of \cite{BIO} for the twisted examples under consideration.

\section*{Acknowledgements}

The authors are grateful to B. Fayad, W. Krieger, S. Matsumoto and N. Watanabe for many useful suggestions, and an anonymous referee whose indications considerably contributed to improve the presentation of the present paper. They also acknowledge the financial support of Italian INDAM-GNAMPA.

\end{document}